\pdfoutput=1
\documentclass[a4paper,leqno]{amsart}
\usepackage{the-Style,the-Macros}

\DeclareMathOperator{\mywedge}{\wt{\wedge}}
\DeclareMathOperator{\mytm}{\wt{\tm}}
\newcommand{\fold}{\mathrm{fold}}

\title{Lie structures in homotopy and isotopy calculi}
\date{January 4, 2026}
\author{Danica Kosanovi\'c}

\address{Universit\"at Bern, Mathematisches Institut (MAI), Alpeneggstrasse 22, 3012 Bern, Switzerland}
\curraddr{CNRS, IMJ-PRG, Université Paris Cité, Sorbonne Université, F-75013 Paris, France}

\email{danica.kosanovic@cnrs.fr}

\begin{document}
\begin{abstract}
    We establish compatibility of Lie structures that appear in homotopy calculus of functors and isotopy calculus of embeddings. On one hand, we give a new proof of the Johnson--Arone--Mahowald result describing the layers of the Goodwillie tower of the identity functor, and we directly compare the spectral Lie bracket with the classical Samelson (Whitehead) bracket on spaces. On the other hand, we geometrically define a bracket on the layers of the embedding calculus tower for embeddings of arcs. These results are unified through the same technical tool, a newly defined bracket on total homotopy fibres of collapsing cubes of wedge sums.
\end{abstract}

\maketitle

%%%%%%%%%%%%%%%%%%%%%%%%%%%%%%%%%%%%%%%%%%%%%%%%%%%%%%%%%%%%%%%%%%
\section{Introduction}

Many parallels can be drawn between two sorts of calculus in topology: the homotopy calculus of functors of Goodwillie~\cite{GooIII}, and the embedding calculus of Weiss, Goodwillie, and Klein~\cite{Weiss,GW,GKmultiple,GKW}. It is likely that these parallels were driving forces in the development of these theories by Goodwillie.
The goal of this paper is to bring forward one of these analogies -- the Lie structure that appears in both -- and its geometric meanings. We explain in what sense the $n$-th stage of the Taylor tower encodes nilpotence information up to degree $n$, in homotopical or isotopical sense, respectively. By extracting and slightly improving known results we hope to point out some connections of interest both to homotopy and isotopy communities. 

In what follows we state our main results, Theorem~\ref{mainthm:A} (about Lie structures in homotopy calculus), Theorem~\ref{mainthm:B} (about Lie structures in isotopy calculus), and Theorem~\ref{mainthm:C} (the bridge between them using Lie structures on total homotopy fibres). We outline related results in Remark~\ref{rem:related}.

%%%%%%%%%%%
\subsection{Homotopy calculus}
Let $\Top_*$ be the topological category (a category enriched in spaces) of well-pointed topological spaces. In homotopy calculus a functor $F\colon \Top_*\to \Top_*$ is approximated by a tower of functors, with the $n$-th stage $P_nF$ for $n\geq1$ given by the degree $\leq n$ polynomial functor ``closest'' to $F$. A functor is polynomial of degree $\leq n$ if it sends a strongly homotopy cocartesian $(n+1)$-cube to a homotopy cartesian $(n+1)$-cube. One can define $P_nF$ by applying to $F$ the homotopy left adjoint of the inclusion of $n$-polynomial into all functors.
For example, for the identity functor $I\colon \Top_*\to \Top_*$ we shall write $P_nX\coloneqq P_nI(X)$, and we have
\[
    P_1X\simeq Q X\coloneqq \Omega^\infty\Sigma^\infty X.
\]
There are several explicit constructions of $P_nX$ including that of Goodwillie~\cite{GooIII}, Eldred~\cite{Eldred}, and Arone--Kankaanrinta~\cite{AK-Snaith}.
The last one allows the authors to show that the limit of the Taylor tower of $I$ is the Bousfield--Kan $\Z$-nilpotent completion:
\[
    P_\infty(X)\simeq \Z_\infty X.
\]
This agrees with $X$ for nilpotent spaces $X$.
%(that is, $\pi_1X$ is nilpotent and acts nilpotently on the stable homotopy of $X$). 
In summary, \emph{Goodwillie's homotopy Taylor tower} of the identity functor $I$ applied to the pointed space $X$ is the tower of spaces
\begin{equation}\label{eq:htpy-tower}
\begin{tikzcd}[column sep=small,row sep=small]
    & & & X\ar{llld}\dar\ar{dl}\ar{drr} &\\
     \Z_\infty X\simeq P_\infty(X)\rar & \cdots\rar & P_{n+1}X\rar & P_nX\rar &\cdots\rar & P_1X\simeq Q(X)
\end{tikzcd}
\end{equation}
interpolating from the stable homotopy of $X$ to the unstable homotopy of its nilpotent completion.

In fact, the connection of the homotopy Taylor tower to nilpotence is much deeper. Biedermann and Dwyer~\cite{Biedermann-Dwyer} redefined some classical notions of homotopy nilpotence so that $\Omega P_n(X)$ is a \emph{homotopy nilpotent group} of class $n$. Related to this, Chorny and Scherer~\cite{Chorny-Scherer} showed that $(n+1)$-fold iterated Samelson brackets of classes in $\pi_*\Omega P_n(X)$ vanish (these are adjoint to Whitehead brackets on $\pi_*P_n(X)$). One defines the Samelson bracket of two maps $f_i\colon Z_i\to\Omega Y$ as the composite
\[\begin{tikzcd}
    Z_1\mywedge Z_2 \rar{f_1\mywedge f_2} & 
    \Omega Y\mywedge \Omega Y \rar{[-,-]} &
    \Omega(Y\vee Y)\rar{\Omega(\fold)} &
    \Omega Y,
\end{tikzcd}
\]
where $\mywedge$ is the homotopy cofibre of $Z_1\vee Z_2\hra Z_1\tm Z_2$ (homotopy equivalent to the cofibre -- the standard smash product $Z_1\wedge Z_2$), the second map is the \emph{commutator bracket}, that sends $(\gamma_1,\gamma_2)\in\Omega Y\tm\Omega Y$ to the commutator loop of $i_1\circ\gamma_1$ and $i_2\circ\gamma_2$, where $i_b\colon Y\hra Y\vee Y$ are the two inclusions, and lastly, $\fold\colon Y\vee Y\to Y$ is the identity map on both wedge factors.

Iterating the Samelson bracket, for a bracketed word $w$ on $n$ letters and maps $f_i\colon Z_i\to\Omega Y_i$ we obtain the iterated Samelson bracket $w(\{f_i\})\colon w(\{Z_i\})\to \Omega (\bigvee_{i=1}^nY_i)$; see~§\ref{subsec:Sam-brackets}. Therefore, the theorem of Chorny and Scherer says that in $\Omega P_n(X)$ ``iterated parameterised commutators of length $(n+1)$'' vanish, and is thus a nilpotence result for homotopy calculus.

In particular, if $x\colon X\to\Omega\Sigma X$ is the map adjoint to the identity and $S$ is a finite set, we have
\begin{equation}\label{eq-intro:can-Sam}
    w(\{x_i\})\colon\; w(\{X\})\ra \Omega\big(\bigvee_S\Sigma X\big),
\end{equation}
the Samelson bracket of $|S|$ copies of $x$, that we call the \emph{canonical Samelson map}. In fact, restricting to words in the Lyndon basis $w\in B\LL(S)$ for the free Lie algebra on $|S|$ letters, these maps completely describe the homotopy type of $\Omega(\bigvee_n\Sigma X)$, thanks to the Hilton--Milnor equivalence~\eqref{eq:HM-all}. 

Words $w\in\LL(n)$ with exactly one appearance of each letter (where $w(\{Z_i\})\simeq Z_1\wedge\dots\wedge Z_n$) form the $n$-th term $\Lie(n)$ of the Lie operad (see §\ref{subsec:Lie}). Precisely those factors of $\Omega(\bigvee_n\Sigma X)$ contribute to the homotopy type of the layers $D_n(X)\coloneqq\hofib(P_nX\to P_{n-1}X)$ of the tower.

Moreover, the aforementioned nilpotence result is in close relation with the result of Johnson~\cite{Johnson} and Ching~\cite{Ching} that the layers $D_nX$ assemble into the \emph{free spectral Lie algebra} on the suspension spectrum of $X$. Namely, for an operad $\mathbf{s}\oLie=\partial_{\bull}I$ that lifts the (suspended) Lie operad to the category of spectra we have 
$D_n(X)= \Omega^\infty((\partial_nI\wedge \Sigma^\infty X^{\wedge n})_{h\Sigma_n})$ (see §\ref{subsec:htpy-calc}). Our first main result makes explicit a connection of Samelson brackets and this operadic Lie structure on the layers.

\begin{mainthm}
\label{mainthm:A}
    For a well-pointed space $X$, the structure of a spectral Lie algebra on the layers of the homotopy Taylor tower of the identity functor at $X$ induces brackets
\[
    [-,-]_{\oLie}\colon\; \Omega D_{n_1}(\Sigma X) \wedge \Omega D_{n_2}(\Sigma X)
    \ra \Omega D_{n_1+n_2}(\Sigma X),
\]
    that generalise Samelson brackets: there are maps $w^{D_n}$ related to $w(\{x_i\})$ from \eqref{eq-intro:can-Sam},
    so that if $w=[w_1,w_2]$ with $w_b\in B\Lie(S_b)$ for $S_1\sqcup S_2=\ul{n}$, $|S_b|=n_b$, then the following square commutes up to homotopy:
    \[\begin{tikzcd}
        \Omega D_{n_1}(\Sigma X) \wedge \Omega D_{n_2}(\Sigma X)
            \rar{ [-,-]_{\oLie}}
        & \Omega D_{n}(\Sigma X)
        \\
        X^{\wedge n_1}
            \wedge
        X^{\wedge n_2}
            \rar{\cong}
            \uar{w_1^{D_{n_1}}\wedge w_2^{D_{n_2}}}
        & X^{\wedge n}
            \uar[swap]{w^{D_n}}\,.
    \end{tikzcd}
    \]
\end{mainthm}
This should be compared to \cite[Proposition~4.28]{Heuts-annals}, proven in a different setting and by a different approach, based on~\cite{Brantner-Heuts}. A more precise description of the maps $w^{D_n}$ will appear in Theorem~\ref{mainthm:C}.

%%%%%%%%%%%
\subsection{Isotopy calculus}
Let now $Man_l$ be the topological category of smooth $l$-dimensional manifolds and codimension $0$ embeddings. In embedding calculus the presheaf $\Emb(-,M)\colon Man_l^{op}\to \Top$ of smooth embeddings for a fixed $d$-dimensional manifold $M$ and $1\leq l\leq d$, is approximated by a tower of presheaves, with the $n$-th stage $T_n(-,M)$ given by the degree $\leq n$ polynomial presheaf closest to $\Emb(-,M)$. This is defined as the right
Kan extension of $i^*\Emb(-,M)$, where $i\colon Disc_{l,\leq n}\to Man_l$ is the full subcategory on manifolds diffeomorphic to $S\tm\R^l$ for a finite set $S$ of cardinality at most $n$. This can be interpreted in terms of cubical diagrams as well~\cite{BW2,Arakawa}.
In particular, when $d-l>0$ the first stage of the Taylor tower is
\[
    T_1(C,M)\simeq\Imm(C,M),
\]
and indeed,  by the ``h-principle'' for immersion theory the immersion functor $\Imm(-,M)$ sends homotopy pullback squares to homotopy pushout squares.

By the key result of Goodwillie and Klein~\cite{GKmultiple}, when $d-l>2$ the maps from $\Emb(C,M)$ to $T_n(C,M)$ become better connected with $n$, implying that
\[
    T_\infty(C,M)\simeq\Emb(C,M).
\]
There is also a version $T_n(C,M)$ for smooth embeddings $\Embp(C,M)$ of manifolds with boundary, with a fixed boundary condition $\partial C\hra\partial M$. In summary, for fixed manifolds $C$ and $M$ with $d-l>2$ the space $\Embp(C,M)$ has the associated \emph{Goodwillie--Weiss Taylor tower} of spaces
\begin{equation}\label{eq:istpy-tower}
\begin{tikzcd}[column sep=small,row sep=small]
    &&& \Embp(C,M)\dar{ev_n}\ar{ld}\ar{rrd}{ev_1=incl}\ar{llld}[swap]{\sim} &\\
    T_\infty(C,M)\rar & \cdots\rar & T_{n+1}(C,M)\rar & T_n(C,M)\rar & \cdots\rar & T_1(C,M)=\Imm(C,M)
\end{tikzcd}
\end{equation}
interpolating from immersions to embeddings. In analogy to the homotopy Taylor tower~\eqref{eq:htpy-tower}, the $n$-th stage of~\eqref{eq:istpy-tower} can be viewed as a ``degree $n$ nilpotent approximation in isotopical sense''. One way this can be formalised is as follows.

We consider the case $C=\D^1$, and $M$ is a manifold of dimension $d\geq3$, and we fix $c\colon\D^1\hra M$. Then the layer $F_{n+1}(M)\coloneqq\hofib(T_{n+1}(\D^1,M)\to T_n(\D^1,M))$ is $(n(d-3)-1)$-connected, the first nontrivial homotopy group was computed in~\cite{K-thesis-paper} as $\pi_{n(d-3)}F_{n+1}(M)\cong \Lie(n)\otm\Z[(\pi_1M)^n]$,
and the corresponding classes in $\pi_{n(d-3)}(\Embp(\D^1,M);c)$ were described in~\cite{K-families} using \emph{graspers}; see §\ref{subsec:graspers}.

\begin{thm}[{\cite{K-families}}]
\label{thm:K-families}
    The grasper map $\realmap_n$ fits into the commutative diagram
    \[\begin{tikzcd}
            \Lie(n)\otm\Z[(\pi_1M)^n] \rar{\realmap_n}\dar{(\ol{ev}_{n+1})_*}[swap]{\cong} 
            &\pi_{n(d-3)}(\Embp(\D^1,M);c) \dar{(ev_{n+1})_*}\rar[two heads]{(ev_n)_*} 
            &\pi_{n(d-3)}T_n(\D^1,M) \dar[equals]\\
            \pi_{n(d-3)}F_{n+1}(M) \rar 
            &\pi_{n(d-3)}T_{n+1}(\D^1,M) \rar[two heads]{(p_{n+1})_*} 
            &\pi_{n(d-3)}T_n(\D^1,M)
    \end{tikzcd}\]
    where the horizontal sequence at the bottom is exact, and the one at the top is exact if $d\geq4$. When $d=3$ the knots in the image of $\realmap_n$ are Gusarov--Habiro clasper $n$-equivalent to the unknot.
    
    For a Lie tree $w\in\Lie(n)$ and $g_i\in\pi_1M$ the family $\realmap_n(w,g_1,\dots,g_n)\colon\S^{n(d-3)}\to\Embp(\D^1,M)$ is adjoint to the Samelson bracket of the classes $m_i^{g_i}\colon\S^{d-1}\to M$ that are meridians of $c$ based using $g_i$.
\end{thm}

    The bottom row is a part of the long exact sequence for the fibration $p_{n+1}$, and when $d\geq4$ the mentioned Goodwillie--Klein convergence~\cite{GKmultiple} implies that $(ev_{n+1})_*$ is an isomorphism, giving the exactness of the top row; neither row is short exact in general.
The last sentence of Theorem~\ref{thm:K-families} says that grasper classes $\realmap_n(w,g_1,\dots,g_n)$ can be viewed as \emph{embedded iterated self-commutators} of $c$ (see \cite{K-families} or Theorem~\ref{thm:emb-comm}). This has inspired us to search for a bracket on the layers of isotopy calculus analogous to $[-,-]_{\Lie}$ from Theorem~\ref{mainthm:A}, and to prove the following analogue of that result.
\begin{mainthm}
\label{mainthm:B}
    For a $d$-dimensional manifold $M$ with boundary and $d\geq3$, the layers of the isotopy Taylor tower of $\Embp(\D^1,M)$ admit brackets
    \[
        [-,-]_{\ext}\colon\; F_{n_1+1}(M) \mywedge F_{n_2+1}(M)\ra F_{n_1+n_2+1}(M),
    \]
    that generalise Samelson brackets: there are maps $w^{F_{n+1}}$ related to $w(\{x_i\})$ from \eqref{eq-intro:can-Sam}, so that if $w=[w_1,w_2]$ with $w_b\in B\Lie(S_b)$ for $S_1\sqcup S_2=\ul{n}$, $|S_b|=n_b$, then the following square commutes up to homotopy:
    \[\begin{tikzcd}
        F_{n_1+1}(M) \mywedge F_{n_2+1}(M)
            \rar{[-,-]_{\ext}}
        & F_{n+1}(M)
        \\
        \Omega^{n_1} (\Sigma^{d-2}(\Omega M)_+))^{\mywedge n_1}
            \mywedge
        \Omega^{n_2} (\Sigma^{d-2}(\Omega M)_+))^{\mywedge n_2}
            \rar{\cong}
            \uar{w_1^{F_{n_1+1}}\mywedge w_2^{F_{n_2+1}}}
        & \Omega^n (\Sigma^{d-2}(\Omega M)_+)^{\mywedge n}
            \uar[swap]{w^{F_{n+1}}} \,.
    \end{tikzcd}
    \]
\end{mainthm}
\begin{rem}
In particular, for the first nonvanishing homotopy group we have
    \[\begin{tikzcd}[row sep=small]
        \pi_{n_1(d-3)}F_{n_1+1}(M) \otm \pi_{n_2(d-3)}F_{n_2+1}(M) \rar{[-,-]_{\ext}} &
        \pi_{n(d-3)}F_{n_1+n_2+1}(M)\\
        \Big(\Lie(n_1)\otm\Z[(\pi_1M)^{n_1}]\Big) \otm
        \Big(\Lie(n_2)\otm\Z[(\pi_1M)^{n_2}]\Big) \rar\uar{\cong} & \Lie(n)\otm\Z[(\pi_1M)^n].\uar{\cong}
    \end{tikzcd}
    \]
The bottom homomorphism sends the pair of ``$\pi_1M$-decorated Lie words'' $(w_1, (g_i)_{i\in\ul{n_1}})\otm (w_2, (g_i)_{i\in\ul{n_2}})$ to the $\pi_1M$-decorated Lie word $([w_1,w_2],(g_i)_{i\in\ul{n}})$ (see Remark~\ref{rem:grafting-Lie}). Therefore, in $\pi_*(\Embp(\D^1,M);c)$ we have a bracket $[-,-]_{\ext}$ only defined on grasper classes: if $w_b\in\Lie(S_b)$ and $g_S\in\pi_1(M^S)$ then
    \[
        [\realmap_{n_1}(w_1, g_{\ul{n_1}}),\;
            \realmap_{n_2}(w_2,g_{\ul{n_2}})]_{\ext}
            =\realmap_n([w_1,w_2],g_{\ul{n}}).\qedhere
    \]
\end{rem}

%%%%%%%%%%%
\subsection{A bridge}

Our bridge between homotopy and isotopy calculi will go through \emph{total homotopy fibres} $\tofib\left(Y_{\ul{n}}; \kappa\right)$ of certain cubical diagrams, called \emph{collapsing cubes}. For fixed spaces $Y_1,\dots,Y_n$ the spaces in the cube are $S\mapsto \bigvee_{i\in S}Y_i$, and the maps $\kappa_{S\sm\{k\}}^S\colon\bigvee_{i\in S}Y_i\to \bigvee_{i\in S\sm\{k\}}Y_i$ collapse a factor.

For sets $S_1,S_2$ such that $S_1\cup S_2=\ul{n}$ (note: here we do not assume they are disjoint) we define in Definition~\ref{def:external-Samelson} a generalisation of the Samelson bracket, called the \emph{external bracket}
\begin{equation}\label{eq-intro:ext-bracket}
    [-,-]_{\ext}\colon \Omega \tofib\Big(\bigvee_{i\in S_1}Y_i;\kappa\Big)\mywedge 
        \Omega \tofib\Big(\bigvee_{i\in S_2}Y_i;\kappa\Big)
           \ra
        \Omega \tofib\Big(\bigvee_{i\in\ul{n}}Y_i;\kappa\Big).
\end{equation}

On one hand, in isotopy calculus for $\Embp(\D^1,M)$ for a $d$-manifold $M$ we know from \cite{K-thesis-paper} that the Taylor tower layers $F_{n+1}(M)\coloneqq\hofib(T_{n+1}(M)\to T_n(M))$ admit homotopy equivalences
\begin{equation}\label{eq-intro:deriv-chi}
    \deriv\chi\colon F_{n+1}(M)\overset{\sim}{\ra} \Omega^{n+1}\tofib
            \Big(\bigvee_{i\in\ul{n}}(\S^{d-1}_i\wedge(\Omega M)_+); \kappa\Big).
\end{equation}
The bracket on these layers from Theorem~\ref{mainthm:B} arises simply from this equivalence and the bracket~\eqref{eq-intro:ext-bracket}.

On the other hand, the layers in homotopy calculus were described by Johnson~\cite{Johnson}, and reinterpreted by Arone--Kankaanrinta~\cite{AK-Snaith} and Arone--Mahowald~\cite{Arone-Mahowald} (see §\ref{subsec:htpy-calc}) in terms of partition complexes $\GPC_{\ul{n}}$ (see Definition~\ref{def:GPC}). In fact, in §\ref{subsec:J} we will use the description of spaces $\GPC_{\ul{n}}$ in terms of weighted trees to directly define a $\Sigma_n$-equivariant map (previously known to exist only indirectly):
    \[
        J\colon\tofib\left(Y_{\ul{n}}; \kappa\right)
        \ra \Map_*(\GPC_{\ul{n}}, \bigwedge_{i\in\ul{n}} Y_i).
    \]
In the key Theorem~\ref{thm:J} we will show that this map $J$ sends external brackets on total homotopy fibres to certain grafting brackets on $\Map_*(\GPC_{\ul{n}},-)$, coming from the cooperadic structure of $\GPC_{\bull}$.

This will allow us to show in Theorem~\ref{thm:J-collapsing} that 
if each $Y_i$ is a suspension the map $J$ is in a sense dual to canonical Samelson maps $w(\{x_i\})$. This together with the Hilton--Milnor theorem implies that $J$ is highly connected (see Corollary~\ref{cor:range}). This gives a different proof -- although using similar ideas -- of the mentioned description of the homotopy calculus layers; see~\eqref{eq:layer-3}. We also show a generalisation in Theorem~\ref{thm:J-general}, for collapsing cubes where $Y_0$ might not be a suspension; this is used in Theorem~\ref{thm:Fn-J} for describing the isotopy calculus layers.
We summarise some of this in the following result, proven in~§\ref{subsec:htpy-calc}.

\begin{mainthm}
\label{mainthm:C}
    For a well-pointed space $X$ and any $w=[w_1,w_2]\in B\Lie(n)$ with $w_b\in\Lie(S_b)$ for $S_1\sqcup S_2=\ul{n}$, the following diagram commutes up to homotopy
    \[\begin{tikzcd}[column sep=large]
        D_{|S_1|}(\Omega\Sigma)(X) \wedge D_{|S_2|}(\Omega\Sigma)(X)
            \rar{[-,-]_{\oLie}}
        & D_n(\Omega\Sigma)(X)
        &[-10pt]\\
        \Omega \Map_*(\GPC_{S_1},(\Sigma X)^{\wedge S_1})\wedge \Omega \Map_*(\GPC_{S_2},(\Sigma X)^{\wedge S_1})
            \rar{[-,-]_{\graft}}
            \uar{\Sigma^\infty()_{D_{S_1}}\wedge\Sigma^\infty()_{D_{S_2}}}
        & \Omega \Map_*(\GPC_{\ul{n}}, (\Sigma X)^{\wedge n})
            \uar[swap]{\Sigma^\infty()_{D_{\ul{n}}}}
        &[-10pt]
        \\
        \Omega \tofib(\vee_{S_1}\Sigma X;\kappa)\mywedge 
        \Omega \tofib(\vee_{S_2}\Sigma X;\kappa)
            \rar{[-,-]_{\ext}}
            \uar{\wt{q}\circ (\Omega J_{S_1}\mywedge \Omega J_{S_2})}
        & \Omega \tofib(\vee_{\ul{n}}\Sigma X;\kappa)
            \uar[swap]{\Omega J_{\ul{n}}}
            \rar{\proj}
        &[-10pt] \Omega\vee_{\ul{n}} \Sigma X
        \\
        X^{\mywedge S_1}
            \mywedge
        X^{\mywedge S_2}
            \rar{\cong}
            \uar{w_1(x_i)^\bull\mywedge w_2(x_i)^\bull}
        & X^{\mywedge \ul{n}}
            \uar[swap]{w(\{x_i\})^\bull}
            \ar{ur}[swap]{w(\{x_i\})}
        &[-10pt]
    \end{tikzcd}
    \]
    In words, there exist \emph{total Samelson maps} $w(\{x_i\})^\bull$ to total homotopy fibres of collapsing cubes, that extend $w(\{x_i\})$, and are also mutually compatible under certain external Samelson brackets $[-,-]_{\ext}$. Moreover, the natural brackets  $[-,-]_{\graft}$  on mapping spaces out of partition complexes are on one hand compatible with $[-,-]_{\ext}$ (under the Johnson map $J$), and on the other hand with the operadic Lie bracket (under stabilisation and homotopy orbits).
\end{mainthm}

Assuming this, we can prove the other two main results.
\begin{proof}[Proof of Theorem~\ref{mainthm:A}]
    Consider the map $X^{\mywedge\ul{n}}\to \Omega D_{n}(\Sigma X)$ defined as the composite of the vertical maps on the right in Theorem~\ref{mainthm:C}. Since the diagram commutes up to homotopy, this is homotopic to a map that factors through the quotient map $\wt{q}\colon X^{\mywedge\ul{n}}\to X^{\wedge\ul{n}}$ to the usual smash product. Thus, we obtain maps
    \[
        w^{D_n}\colon X^{\wedge\ul{n}}\to \Omega D_{n}(\Sigma X),
    \]
    and observe that for them the square in Theorem~\ref{mainthm:A} commutes thanks to Theorem~\ref{mainthm:C}.
\end{proof}
\begin{proof}[Proof of Theorem~\ref{mainthm:B}]
    We define the bracket
    \[
        [-,-]_{\ext}\colon F_{n_1+1}(M) \mywedge F_{n_2+1}(M)\to F_{n_1+n_2+1}(M)
    \]
    by $f\mywedge g\mapsto (\deriv\chi)^{-1}[\deriv\chi(f),\deriv\chi(g)]_{\ext}$ using the external bracket~\eqref{eq-intro:ext-bracket} and 
    the map $\deriv\chi$ from \eqref{eq-intro:deriv-chi} (see also Theorem~\ref{thm:K-thesis-paper}). 
    Furthermore, if in Theorem~\ref{mainthm:C} we put $X\coloneqq\S^{d-2}\wedge(\Omega M)_+$, then we can define
    \[\begin{tikzcd}
    w^{F_{n+1}}\colon
    \Omega^n w(\Sigma^{d-2}(\Omega M)_+)\rar{\Omega^nw(\{x_i\})} &
    \Omega^{n+1}\tofib\Big(\bigvee_{\ul{n}}(\S^{d-1}\wedge(\Omega M)_+); \kappa\Big) \rar{(\deriv\chi)^{-1}}
    & F_{n+1}(M).
    \end{tikzcd}
    \]
    For these maps the square in Theorem~\ref{mainthm:B} commutes thanks to Theorem~\ref{mainthm:C}.
\end{proof}

\begin{rem}
    We mentioned that in~\eqref{eq-intro:ext-bracket} we can take $S_1\cap S_2\neq\emptyset$. These external brackets could give interesting and more general brackets in Theorem~\ref{mainthm:B}. However, in Theorem~\ref{mainthm:A} the situation is different: the spectral and grafting brackets deal only with words with non-repeating letters, as this is the only part of the Hilton--Milnor homotopy type that is visible in the homotopy layers. We leave the more general external brackets for isotopy layers for later investigation.
\end{rem}

%%%%%%%%%%%
\begin{rem}\label{rem:related}
We mention some related results. Malin~\cite{Malin-Koszul-self-manifolds} has recently proved a conjecture of Ching \cite{Ching-skye}, that had predicted an equivalence between two Lie structures on the stabilised configuration spaces $\Sigma^\infty\Conf_n(M)$ for a framed manifold $M$. One is the spectral Lie structure from homotopy calculus (on the derivatives $((M^+)^{\wedge n}/\mathrm{fat diagonal})^\vee\simeq\Sigma^{-dn}\Sigma^{\infty}\Conf_n(M)$ of the functor $\Sigma^\infty\Map_*(M^+,-)$), the other is the restriction of the right $E_d$-structure from isotopy calculus (using Koszul self-duality of $E_d$).

In another direction, Munson~\cite{Munson} compared Johnson's map $J$ to Koschorke's invariants in the setting of link maps. Closely related is the work of Klein and Williams~\cite{Klein-Williams-III} on links and link maps.

Total homotopy fibres of collapsing cubes were also studied by Rognes~\cite{Rognes}. When all factors are spheres, we recover and correct \cite[Corollary~13.9]{Rognes}, that says that the first nonvanishing homotopy group admits a $\Sigma_n$-equivariant isomorphism
    \[
        \pi_{n(d-2)}\Omega\tofib\big(\bigvee_{\ul{n}}\S^{d-1}; \kappa\big)
        \cong\Lie_{d-2}(n),
    \]
for $\Lie_{d-2}(n)\cong\Lie(n)\otm(\sgn_n)^{\otm d-2}$, the sign-twisted Lie representation; see Corollary~\ref{cor:ex-Lie}.

For more on partition complexes and operads see also \cite{Fresse-partition,Valette-partition,Arone-Brantner}.
\end{rem}

%%%%%%%%%%%
\subsection*{Outline}
    In §\ref{sec:background} we mainly review \textit{background}. 
    We recall Lie trees and partition complexes in §\ref{subsec:Lie}, Goodwillie partition complexes and grafting brackets in §\ref{subsec:grafting}, and Samelson brackets in §\ref{subsec:Sam-brackets}.

    In §\ref{sec:cubes-ext} we discuss \textit{cubes}. 
    After briefly introducing them and their total homotopy fibres in §\ref{subsec:cubes}, 
    we present new material in §\ref{subsec:ext-Sam-brackets}: we define external Samelson brackets, and in §\ref{subsec:examples} we discuss some examples. Using these brackets we define total Samelson maps in §\ref{subsec:total-Sam}.

    In §\ref{sec:improve} we discuss our \textit{core results}.
    In §\ref{subsec:J} we define the map $J$ and prove the key Theorem~\ref{thm:J}. In §\ref{subsec:J-collapsing} we prove Theorem~\ref{mainthm:C}, except for the commutativity of the top square. In §\ref{subsec:J-general} we prove a generalisation, for collapsing cubes of not only suspensions.

    In §\ref{sec:calculi} we derive our \textit{applications}:
    in §\ref{subsec:htpy-calc} we explain how this relates to homotopy calculus and prove the rest of Theorem~\ref{mainthm:C}, 
    and in §\ref{subsec:istpy-calc} we summarise isotopy calculus for arcs and define graspers.

%%%%%%%%%%%
\subsection*{Acknowledgements}
    I am much obliged to the referees for their extremely valuable feedback, that corrected several errors and greatly improved the exposition. I thank the editors for the offer to contribute.
    This work was partially supported by the Science Fund of the Republic of Serbia (grant number 7749891).

%%%%%%%%%%%%%%%%%%%%%%%%%%%%%%%%%%%%%%%%%%%%%%%%%%%%%%%%%%%%%%%%%%
\section{Background}\label{sec:background}

%%%%%%%%%%%
\subsection{Trees}
\label{subsec:Lie}

We discuss Lie trees and partition complexes.

\subsubsection{Lie trees}

Let $\LL(n)$ be the free Lie ring over $\Z$ on $n$ letters $x^1,\dots,x^n$. This is the free abelian group generated by \emph{bracketed words}, modulo the antisymmetry and Jacobi relations:
\begin{equation}\label{eq:L-as-ihx}
    [w,v]+[v,w]=0,\qquad [v,[w,z]]-[[v,w],z]-[w,[v,z]]=0.
\end{equation}
The free Lie algebra $\LL(n)\otm\Q$ has a basis given by the set $B\LL(n)$ of Lyndon words on $n$ letters. We still have an inclusion $B\LL(n)\subset \LL(n)$, but the 2-torsion elements $[v,v]$ are not in the subgroup generated by the image. %(Some authors call this a quasi-Lie ring, and instead require that $[v,v]=0$ in a Lie ring.)
We will also need the fact that if $w=[w_1,w_2]$ is in $B\LL(n)$, then both $w_i\in B\LL(n_i)$ for some $n_1+n_2=n$. 

Define $\Lie(n)\subset \LL(n)$ as the subgroup spanned by all bracketed words that contain each letter $x^i$ exactly once. % Equivalently, it is the intersection of the weight $n$ subgroup of $\LL(n)$ with the kernels of all the maps that send one letter to zero. 
Note that such words correspond to \emph{Lie trees} with: ($\diamond$) $n+1$ univalent vertices together with labels that are in bijection with $[n]\coloneqq\{0,1,\dots,n\}$, and ($\diamond$) $n-1$ trivalent vertices together with a cyclic order of incident edges. From such a tree $\Gamma$ we read off a bracketed word $w_\Gamma$ inductively: if univalent vertices of $\Gamma$ are only $\{0,i\}$ then $w_\Gamma\coloneqq x^i$, and otherwise $0$ in $\Gamma$ is adjacent to a trivalent vertex, at which two ordered daughter trees meet, and we let $w_\Gamma=[w_{\Gamma_1},w_{\Gamma_2}]$. 

Therefore, we can think of $w\in\Lie(n)$ as a bracketed word and a Lie tree interchangeably, and $\Lie(n)$ is equivalently described as the quotient of the free abelian group $\Z[\Tree(n)]$, modulo the relations that correspond to \eqref{eq:L-as-ihx}, locally given by
\[
    \begin{tikzpicture}[baseline=2ex,scale=0.3,every node/.style={scale=0.8}]
            \clip (-1.7,-0.6) rectangle (1.7,3.6);
            \draw
                (0,-0.3) node[]{$\vdots$};
            \draw[thick]
                (0,0) -- (0,1) -- (-1.25,2.25) node[pos=1,above]{$\Gamma_2$}
                        (0,1) -- (1.25,2.25)  node[pos=1,above]{$\Gamma_1$};
    \end{tikzpicture} 
    + 
    \begin{tikzpicture}[baseline=2ex,scale=0.3,every node/.style={scale=0.8}]
            \clip (-1.7,-0.6) rectangle (1.7,3.6);
            \draw (0,-0.3) node[]{$\vdots$};
            \draw[thick]
                (0,0) -- (0,1) -- (-1.25,2.25) node[pos=1,above]{$\Gamma_1$}
                        (0,1) -- (1.25,2.25)   node[pos=1,above]{$\Gamma_2$};
    \end{tikzpicture} 
    =\: 0,\quad\quad
    \begin{tikzpicture}[baseline=2ex,scale=0.25,every node/.style={scale=0.8}]
            \clip (-2.5,-0.6) rectangle (2.5,4.45);
            \draw (0,-0.3) node[]{$\vdots$};
            \draw[thick]
                (0,0) -- (0,1) -- (-1,2) -- (-2,3) node[pos=1,above]{$\Gamma_3$}
                                  (-1,2) -- (0,3) node[pos=1,above]{$\Gamma_2$}
                        (0,1) -- (2,3)  node[pos=1,above]{$\Gamma_1$};
    \end{tikzpicture} 
    -
    \begin{tikzpicture}[baseline=2ex,scale=0.25,every node/.style={scale=0.8}]
            \clip (-2.5,-0.6) rectangle (2.5,4.45);
            \draw (0,-0.3) node[]{$\vdots$};
            \draw[thick]
                (0,0) -- (0,1) -- (-1,2) -- (-2,3) node[pos=1,above]{$\Gamma_3$}
                        (1,2) -- (0,3) node[pos=1,above]{$\Gamma_2$}
                        (0,1) -- (2,3)  node[pos=1,above]{$\Gamma_1$};
    \end{tikzpicture} 
    + 
    \begin{tikzpicture}[baseline=2ex,scale=0.25,every node/.style={scale=0.8}]
            \clip (-2.5,-0.6) rectangle (2.5,4.45);
            \draw (0,-0.3) node[]{$\vdots$};
            \draw[thick]
                (0,0) -- (0,1) -- (-1,2) -- (-2,3) node[pos=1,above]{$\Gamma_1$}
                                  (-1,2) -- (0,3) node[pos=1,above]{$\Gamma_3$}
                        (0,1) -- (2,3)  node[pos=1,above]{$\Gamma_2$};
    \end{tikzpicture} 
    =\: 0.
\]
The subset $B\Lie(n)\coloneqq B\LL(n)\cap \Lie(n)$ consists of right-normed words 
\begin{equation}\label{eq:right-normed}
    w_\sigma\coloneqq[x^{\sigma(1)},\dots,[x^{\sigma(2)},[x^{\sigma(n-1)},x^n]]]
\end{equation}
for $\sigma\in\Sigma_{n-1}$, which form a basis for $\Lie(n)$, so $\Lie(n)\cong\Z^{(n-1)!}$.
Note that $\Sigma_n$ acts on $\Lie(n)$ by permuting $x^i$'s, and this restricts to the regular $\Z$-representation of $\Sigma_{n-1}$. However, the $\Sigma_n$-representation $\Lie(n)$ is more complicated; it is the arity $n$ of the \emph{Lie operad} in abelian groups, whose algebras are Lie rings over~$\Z$.

We will also need a variation of this in the graded world (see for example Corollary~\ref{cor:ex-Lie}). Let each $x^i$ have degree $|x^i|=D$ for some $D\geq0$, and let $\LL_D(\ul{n})$ be the free abelian group generated by bracketed words, modulo the \emph{graded} antisymmetry $[w,v]+(-1)^{|v||w|}[v,w]=0$ and Jacobi relations $[v,[w,z]]-[[v,w],z]- (-1)^{|v||w|}[w,[v,z]]=0$. Define $\Lie_D(n)\subset\LL_D(\ul{n})$ as the n-linear part.
\begin{lem}[\cite{Conant,Robinson}]
\label{lem:Lie-d(n)}
    If $\sgn_n=\Z$ denotes the sign representation of $\Sigma_n$, then there are $\Sigma_n$-isomorphisms $\Lie(n)\to\Lie_D(n)$ for $D$ even and $\Lie(n)\otm\sgn_n\to\Lie_D(n)$ for $D$ odd. In short,
    \[
        \Lie_D(n)\cong \Lie(n)\otm\sgn_n^{\otm D}.
    \]
\end{lem}

\subsubsection{Partition complexes}

Let $\Pi_{\ul{n}}$ be the poset of all partitions of $\ul{n}\coloneqq\{1,\dots,n\}$ excluding the minimal (discrete) and maximal (trivial) partitions. The geometric realisation of its nerve is the \emph{partition complex} $|\Pi_{\ul{n}}|=|N\Pi_{\ul{n}}|$. This has an action of $\Sigma_n$ induced from the action on $\ul{n}$. Note that $\Pi_{\ul{1}}$ and $\Pi_{\ul{2}}$ are empty. 

\begin{thm}[\cite{Barcelo,AK-Lie,Wachs,Robinson}]
\label{thm:Robinson}
    For $n\geq3$ there is a non-equivariant homotopy equivalence
    \begin{equation}\label{eq:Pin-nonequiv}
        |\Pi_{\ul{n}}|\simeq\bigvee_{(n-1)!} \S^{n-3}
    \end{equation}
    On the other hand, there is an isomorphism of $\Sigma_n$-representations 
    $\wt{H}^{n-3}(|\Pi_{\ul{n}}|;\Z)\cong \Lie(n)\otm\sgn_n$.
\end{thm}
\begin{proof}[Proof Idea]
    A point in $|\Pi_{\ul{n}}|$ is described (non-uniquely) by a pair 
\begin{equation}\label{eq:Pi-n-pt}
    (\wh{0}<b_0\leq b_1\leq\dots\leq b_k<\wh{1},\;
    0\leq t_1\leq t_2\leq\dots\leq t_k\leq1)
\end{equation}
    of a chain of partitions $b_i\in\Pi_{\ul{n}}$ and a point in the $k$-simplex $\Delta^k$. 
    % $|\Pi_{\ul{n}}|$ is the complex of non-connected graphs on $n$ vertices.
    Using this one can show that $|\Pi_{\ul{n}}|$ is homeomorphic to the space of \emph{fully grown weighted trees} with: ($\diamond$) $n+1$ univalent vertices labelled by $[n]$, ($\diamond$) some vertices of valence $\geq3$, and ($\diamond$) a positive length $0<l_e\leq1$ of each internal edge $e$, so that $l_e=1$ for at least one edge. See \cite[Proposition~2.7]{Robinson} for a proof; roughly, to obtain a point as in \eqref{eq:Pi-n-pt} we first rescale a fully grown weighted tree to have total height one, and then determine at which heights $t_i$ the internal vertices occur, and the partition of the leaves above that height. 

    It is not hard to show that $N\Pi_{\ul{n}}$ is a simplicial set of dimension $n-3$, with top simplices corresponding to fully grown weighted \emph{Lie} trees. Let $L$ denote the union of the interiors of those $(n-3)$-simplices that correspond to right-normed Lie trees $w_\sigma$ from~\eqref{eq:right-normed}. There is $(n-1)!$ many of them, and one can show that $N\Pi_{\ul{n}}\sm L$ is contractible. Therefore, \eqref{eq:Pin-nonequiv} holds. Note that it forgets the natural $\Sigma_n$-action on $|\Pi_{\ul{n}}|$, and just remembers the regular $\Sigma_{n-1}$-action.

    In particular, it follows that a basis of $\wt{H}^{n-3}(|\Pi_{\ul{n}}|;\Z)$ is given by cocycles $c_\sigma$ defined by $c_\sigma(w_\sigma)=1$ and zero on all other $(n-3)$-simplices.
    In order to determine the $\Sigma_n$-action on $\wt{H}^{n-3}(|\Pi_{\ul{n}}|;\Z)$ Robinson~\cite[Section~3]{Robinson} directly defines a homomorphism $\wt{H}^{n-3}(|\Pi_{\ul{n}}|;\Z)\to\Lie(n)\otm\sgn_n$ which sends $c_\sigma$ to $w_\sigma$ and is $\Sigma_n$-equivariant. This sends a basis to a basis, so has to be an isomorphism.
\end{proof}

%%%%%%%%%%%
\subsection{Grafting}
\label{subsec:grafting}

In this section we recall the definition of Goodwillie partition complexes $\GPC_{\ul{n}}$ and the ungrafting map of trees. This not only defines a cooperadic structure on these spaces, but we can also use this map to define grafting brackets $[-,-]_{\graft}$ on mapping spaces out of $\GPC_{\ul{n}}$.

\subsubsection{Goodwillie partition complexes}
We choose this name for the following well-studied object. Recall that the join of two spaces $X* Y$ is the homotopy pushout of $X\leftarrow X\tm Y\ra Y$.

\begin{defn}\label{def:GPC}
     Define the \emph{Goodwillie partition complex} as $\GPC_{\ul{n}}\coloneqq \S^1*|\Pi_{\ul{n}}|$ if $n\geq2$, and let $\GPC_{\ul{1}}\coloneqq\S^0$. More generally, for a finite set $S$ with $|S|\geq2$ and its poset $\Pi_S$ of all partitions we have $\GPC_{S}\coloneqq \S^1*|\Pi_S|$.
\end{defn}
\begin{rem}
    There is a homotopy equivalence $\S^1*|\Pi_{\ul{n}}|\simeq \S^1\wedge (\S^0*|\Pi_{\ul{n}}|)$, the reduced suspension of the unreduced suspension, and the latter is often used as the definition of $\GPC_{\ul{n}}$. See~\cite[Lemma~2.17]{Arone-Brantner}.
\end{rem}

In particular, $\GPC_{\ul{n}}$ has a $\Sigma_n$-action coming from the action on $\ul{n}$, and by Theorem~\ref{thm:Robinson} it admits a non-equivariant equivalence $\GPC_{\ul{n}}\simeq\bigvee_{(n-1)!}\S^{n-1}$, for any $n\geq1$, and an isomorphism of $\Sigma_n$-representations
\begin{equation}\label{eq:coh-Tn}
    \wt{H}^{n-1}(\GPC_{\ul{n}};\Z)\cong \wt{H}^{n-3}(|\Pi_{\ul{n}}|;\Z)\cong\Lie(n)\otm\sgn_n.
\end{equation}

Similarly as in the proof of Theorem~\ref{thm:Robinson}, $\GPC_{\ul{n}}$ can be viewed as a space of weighted trees $(\Gamma,l)$ with $n$ leaves, with a function $l$ assigning to each edge a length in $[0,1]$, so that the total distance from the root to each leaf is exactly $1$; if a root or leaf edge has length zero, the weighted tree is identified with the basepoint. For example, $\GPC_{\ul{2}}\coloneqq\S^1*\emptyset\cong\S^1$ is the unique tree with two leaves, with the length of the root edge $l(e_0)\in\S^1$ as the corresponding value in $\S^1$. The action of $\Sigma_2$ on $\GPC_{\ul{2}}$ is trivial.

\subsubsection{The cooperad structure}

In fact, in \cite[Lemma~8.6]{Ching} Ching identifies $\GPC_{\ul{n}}=B(com)(n)$,
the arity $n$ of the bar construction of the commutative operad in pointed spaces: $Com(n)=\S^0$ for every $n\geq1$. Then $B(com)(n)=|B_{\bull}(com)(n)|$ has as $k$-simplices chains of partitions $(\wh{0}\leq b_0\leq b_1\leq\dots\leq b_k\leq \wh{1})$, 
with $b_i\in\Pi_{\ul{n}}$, so that if $b_0\neq\wh{0}$ or $b_k\neq\wh{1}$ the chain is identified with the basepoint.
Compare this to \eqref{eq:Pi-n-pt}, where we strictly disallow these equalities; this amounts to suspensions that take $|\Pi_{\ul{n}}|$ to $\GPC_{\ul{n}}$. 

Moreover, Ching constructs a cooperad structure on $B(com)$, given by ``ungrafting'' weighted trees. In particular, for every $\ul{n}=S_1\sqcup S_2$ there is a map
\begin{equation}\label{eq:cooperad}
    \ungraft_{S_1,S_2}\colon\; \GPC_{\ul{n}}\ra \GPC_{\ul{2}}\wedge \GPC_{S_1}\wedge\GPC_{S_2}\cong\S^1\wedge \GPC_{S_1}\wedge\GPC_{S_2},
\end{equation}
that sends a weighted tree with $n$ leaves to the basepoint unless it is of the shape $([\Gamma_1,\Gamma_2],l)$ for trees with leaves $S(\Gamma_1)=S_1$, $S(\Gamma_2)=S_2$, or $S(\Gamma_1)=S_2$, $S(\Gamma_2)=S_1$; in those two cases the tree is in turn sent to the tuple $(l(e_0),(\Gamma_1,l_1),(\Gamma_2,l_2))$ or $(l(e_0),(\Gamma_2,l_2),(\Gamma_1,l_1))$ respectively, where $e_0$ is the root edge of length $l(e_0)\in [0,1]/\partial\cong\S^1$, and an edge $e$ in $\Gamma_b$ has length $l_b(e)\coloneqq \tfrac{l(e)}{1-l(e_0)}$. See also \cite[Definition~4.19]{Arone-Brantner} for an exposition of ungrafting, also more generally.

The cooperad structure~\eqref{eq:cooperad} on the collection $\GPC_{\ul{n}}$, $n\geq1$, implies that the collection of Spanier--Whitehead duals forms an operad in spectra:
\begin{equation}\label{eq-def:s-spectral-Lie}
    \mathbf{s}\oLie(n) \coloneqq (\Sigma^\infty\GPC_{\ul{n}})^\vee
\end{equation}
where $\Sigma^\infty$ takes the suspension spectrum, and $^\vee$ the Spanier--Whitehead dual. In fact, the operad $\mathbf{s}\oLie$ is precisely the Koszul dual of the commutative operad $Com$ in the category of spectra~\cite{Ching},  $\mathbf{s}\oLie=B(com)^\vee$, and is thus called the \emph{(shifted) spectral Lie operad}. 
Indeed, this gives an extension to spectra of the theory of Lie algebras: %It is equivalently described as the cobar construction on the cooperad $Com^\vee$, the Spanier--Whitehead dual of the commutative operad.
by~\eqref{eq:coh-Tn} its homology is concentrated in degree $1-n$ where
\[
    H_{1-n}(\mathbf{s}\oLie(n)) \cong \Lie(n)\otm \sgn_n,
\]
the arity $n$ of the Lie operad, shifted by $1-n$ and twisted by the sign representation.

It makes sense to also consider the operad obtained by operadic desuspension:
\begin{equation}\label{eq-def:spectral-Lie}
    \oLie(n)\coloneqq\Sigma^{1-n}(\Sigma^\infty\GPC_{\ul{n}})^\vee,
\end{equation}
that we call the \emph{spectral Lie operad}. Its homology is the Lie operad itself: $H_0(\oLie(n)) \cong \Lie(n)$.

% We will need the following straightforward lemma.
% % We write $\Delta^S$ for the simplex with the vertex set $S$, so this is homeomorphic to a $|S|-1$-dimensional simplex.

% \begin{lem}
% \label{lem:delta-to-part}
%     For $\sigma\in\Sigma_{n-1}$ there are maps $\Gamma_\sigma\colon \Delta^{n-1}\to\GPC_{\ul{n}}$ that send $(t_1,\dots,t_n)$ with $t_1+\dots+t_n=1$ to the slanted tree $w_\sigma$ as in \eqref{eq:right-normed}, with the top leaf $n$, and the root length $l(e_0)=t_{\sigma 1}$, the next inner edge length $t_{\sigma 2}$, and so on, so that the last inner edge has length $t_{\sigma(n-1)}$. 
%     Moreover, there is a projection $pr_\sigma\colon\GPC_{\ul{n}}\to\bigvee_{\sigma\in\Sigma_{n-1}}\S^{n-1}\to\S^{n-1}$ to the $\sigma$-summand, so that $pr_\sigma\circ\Gamma_\sigma$ is the quotient by the boundary.
% \end{lem}

\subsubsection{Grafting brackets}

The cooperad structure~\eqref{eq:cooperad} on the Goodwillie partition complexes gives rise to a bracket on maps out of them into smash products, as follows.
\begin{defn}\label{def:graft-bracket}
    For spaces $Z_1,Z_2$ and sets $S_1\sqcup S_2=\ul{n}$ define the \emph{grafting bracket} $[-,-]_{\graft}$ as
\[\begin{tikzcd} 
    \Map_*(\S^1\wedge\S^1\wedge\GPC_{S_1}\wedge\GPC_{S_2},Z_1\wedge Z_2)
        \rar{\phi^*}[swap]{\sim} &
    \Map_*(\S^1\wedge\GPC_{\ul{2}}\wedge\GPC_{S_1}\wedge\GPC_{S_2},Z_1\wedge Z_2)
        \dar{\ungraft_{S_1,S_2}^*}
    \\
    \Map_*(\Sigma\GPC_{S_1},Z_1)\wedge \Map_*(\Sigma\GPC_{S_2},Z_2)
        \uar{-\wedge-}
    & \Map_*(\Sigma\GPC_{\ul{n}},Z_1\wedge Z_2)
    \\[-5pt]
    \Omega\Map_*(\GPC_{S_1},Z_1)\wedge \Omega\Map_*(\GPC_{S_2},Z_2)
        \rar[densely dashed]{[-,-]_{\graft}}
        \uar{\cong}
    & \Omega\Map_*(\GPC_{\ul{n}},Z_1\wedge Z_2)
        \uar{\cong}
\end{tikzcd}
\]
    where the homeomorphism $\phi\colon\S^1\wedge\GPC_{\ul{2}}\to\S^1\wedge\S^1$ is given by $\phi(\theta\wedge t)=t \wedge 2\theta t$ if $\theta\in[0,1/2]$, and $\phi(\theta\wedge t)=(2-2\theta)t\wedge t$ if $\theta\in[1/2,1]$. 
\end{defn}
    In other words, for $\psi_b\in \Map_*(\GPC_{S_b},\Omega Z_b)$ the map $[\psi_1,\psi_2]_{\graft}\in\Map_*(\GPC_{\ul{n}},\Omega(Z_1\wedge Z_2))$ sets to $*$ any weighted tree that is not a grafting $([\Gamma_1,\Gamma_2],l)$ of trees with $S_1$ and $S_2$ as leaf sets. Otherwise, we let $T_0=l(e_0)$ and $\gamma_b=\psi_b(\Gamma_b,l_b)$ and define
\[
    [\psi_1,\psi_2]_{\graft}\colon \theta\wedge([\Gamma_1,\Gamma_2],l) \mapsto
     \left\{\begin{array}{{@{}ll@{}}}
            \gamma_1(T_0) \wedge  \gamma_2(2\theta T_0),    & \theta\in[0,\tfrac{1}{2}],\\[3pt]
            \gamma_1((2-2\theta)T_0) \wedge \gamma_2(T_0),  & \theta\in[\tfrac{1}{2},1].
        \end{array}\right.
\]

The following is an easier, but less symmetric version of the grafting bracket; see also Example~\ref{ex:cube-deg-2}.
\begin{lem}\label{lem:square-htpy}
    The map $\phi$ is homotopic to the map $\theta\wedge T_0 \mapsto T_0\wedge\theta$.
    % The map $\phi^*(\psi_1\wedge\psi_2)\colon\S^1\wedge\GPC_{\ul{2}}\to \Map_*(\GPC_{S_1}\wedge\GPC_{S_2},Z_1\wedge Z_2)$ is homotopic to the map given by $\theta\wedge T_0\mapsto \psi_1(\Gamma_1,l_1)(T_0)\wedge\psi_2(\Gamma_2,l_2)(\theta)$}. 
    Thus, up to homotopy we have
    \[
        [\psi_1,\psi_2]_{\graft}\colon \theta\wedge([\Gamma_1,\Gamma_2],l)\mapsto
        \psi_1(\Gamma_1,l_1)(T_0)\wedge \psi_2(\Gamma_2,l_2)(\theta).
    \]
\end{lem}
\begin{proof}
    To homotope $\phi$ to the swap map, first straighten out the line $(2-2\theta)T_0 \wedge T_0$ for $\theta\in[1/2,1]$ into $T_0\wedge (2-2\theta)T_0$ by an easy homotopy, and then rescale.
\end{proof}

%%%%%%%%%%%
\subsection{Samelson brackets and canonical maps}
\label{subsec:Sam-brackets}

In this section we recall the definition of Samelson brackets on loop spaces, that are adjoint to Whitehead brackets on spaces; see for example~\cite{Whitehead}. The collection of canonical Samelson maps describes the homotopy type of the loop space of a wedge sum of suspensions -- this is the Hilton--Milnor equivalence that we recall in \eqref{eq:HM-all}.

Samelson brackets are usually only defined on homotopy classes, but we will make a choice of particular representatives. For this we need to first introduce fat-smash products.

%%%%%%%%%%%
\subsubsection{On smash products}\label{subsec:smash}
For well-pointed spaces $Z_1$ and $Z_2$ let $\mathrm{v}\colon Z_1\vee Z_2\cong Z_1\tm\{*_2\}\vee\{*_1\}\tm Z_2\hra Z_1\tm Z_2$ be the canonical inclusion.
The \emph{smash product} of $Z_1$ and $Z_2$ is defined as the cofibre of $\mathrm{v}$, that is, the quotient $Z_1\wedge Z_2\coloneqq Z_1\tm Z_2/ \mathrm{v}(Z_1\vee Z_2)$. Let $q\colon Z_1\tm Z_2\sra Z_1\wedge Z_2$ be the quotient map.

\begin{defn}\label{def:my-smash}
    Define the \emph{fat-smash product} of $Z_1$ and $Z_2$ as the homotopy cofibre of $\mathrm{v}$, that is, its (unreduced) mapping cone:
    \[
    Z_1\mywedge Z_2\coloneqq \frac
        {(Z_1\vee Z_2)\tm[0,1] \cup_{(z,1)\sim \mathrm{v}(z)} (Z_1\tm Z_2)}
        {(Z_1\vee Z_2)\tm\{0\}}
    .\qedhere
    \]
\end{defn}
There is a quotient map $\wt{q}\colon Z_1\mywedge Z_2\sra Z_1\wedge Z_2$ that collapses the cone (the ``nullhomotopies of the wedge''), and this is a homotopy equivalence. The fat-smash product is useful in the following situations.
\begin{lem}\label{lem:lift}
    Suppose $f$ is a map out of $Z_1\tm Z_2$ that is not constant on $Z_1\vee Z_2$, but is nullhomotopic on it. A choice of a nullhomotopy precisely corresponds to a lift of $f$ along the inclusion $\wt{i}\colon Z_1\tm Z_2\hra Z_1\mywedge Z_2$.
\end{lem}

Similarly, the smash product of pairs of spaces is defined as $(Z_1,A_1)\tm(Z_2,A_2)=(Z_1\tm Z_2, A_{\tm})$ for $A_{\tm}\coloneqq A_1\tm Z_2\cup_{A_1\tm A_2} Z_1\tm A_2$ with $\mathrm{v}\colon A_{\tm}\hra Z_1\tm Z_2$, and we define their \emph{fat-smash product} as
\[
    (Z_1,A_1)\mytm(Z_2,A_2)\coloneqq \Big( 
    \big(A_{\tm}\tm[0,1]\big)\cup_{(z,1)\sim \mathrm{v}(z)} \big(Z_1\tm Z_2\big),\;
    A_{\tm}\tm\{0\}\Big).
\]
We have analogous maps
\[
\begin{tikzcd}[column sep=tiny,row sep=small]
    (Z_1\tm Z_2,\emptyset) \ar[hook]{rr}{\wt{i}}\ar[two heads]{dr}[swap]{q} && (Z_1,A_1)\mytm(Z_2,A_2)\ar[two heads]{dl}{\wt{q}} \\
    & (Z_1,A_1)\tm(Z_2,A_2) &
\end{tikzcd}
\]
\begin{lem}\label{lem:naturality}
    For two maps of pairs $f_b\colon (Z_b,A_b)\to (V_b,B_b)$ the smash product $f_1\tm f_2\colon (Z_1\tm Z_2, A_{\tm})\to (V_1\tm V_2, B_{\tm})$ admits a canonical lift 
    $f_1\mywedge f_2\colon (Z_1,A_1)\mywedge(Z_2,A_2)\to (V_1,B_1)\mywedge (V_2,B_2)$.
\end{lem}
\begin{proof}
    We have a canonical choice of a nullhomotopy of $(f_1\tm f_2)|_{A_{\tm}}$, that sends $((z_1,z_2),t)\in A_{\tm}\tm[0,1]$ to $((f_1(z_1),f_2(z_2)),t)\in B_{\tm}\tm[0,1]$.
\end{proof}

\begin{lem}\label{lem:equivalent-to-pairs}
    A map $Z_1\mywedge Z_2\to V$ is equivalent to a map of pairs $(Z_1,*_1)\mytm(Z_2,*_2)\to (V,*_V)$. % Moreover, to specify such a map it suffices to specify a map $f\colon Z_1\tm Z_2\to V$ and paths from $*_V$ to $f(z_1,*)$ and $f(*,z_2)$  for all $z_1,z_2$.
\end{lem}
\begin{proof}
    This is because $(Z_1,*_1)\mytm(Z_2,*_2)=((Z_1\vee Z_2)\tm[0,1]\cup_a Z_1\tm Z_2, (Z_1\vee Z_2)\tm\{0\})$.
\end{proof}

\subsubsection{Commutator brackets}

Fix spaces $Z_1$ and $Z_2$, and denote $\iota_b\colon Z_b\hra Z_1\vee Z_2$ for $b=1,2$. We simplify notation by writing $\gamma_b$ for $\iota_b(\gamma_b)$. Consider the map
\[
    (\gamma_1,\gamma_2)\in\Omega Z_1\tm\Omega Z_2
    \quad\mapsto\quad
    \gamma_1\gamma_2\gamma_1^{-1}\gamma_2^{-1}\in\Omega(Z_1\vee Z_2).
\]
On $\Omega Z_1\tm*\subseteq\Omega Z_1\tm\Omega Z_2$ and $*\tm\Omega Z_2\subseteq\Omega Z_1\tm\Omega Z_2$ we respectively have $[\gamma_1,\const_*]$ and $[\const_*,\gamma_2]$. For these we have the following canonical nullhomotopies.
\begin{defn}\label{def:canonical}
    Given loops $\gamma_b\in\Omega Z_b$ for $b=1,2$ we define the \emph{canonical nullhomotopies}
    \begin{align*}
    [\gamma_1,\const_*]^{\downarrow \bull}\colon[0,1]\to\Omega(Z_1\vee Z_2)
    \text{ from }[\gamma_1,\const_*]^{\downarrow 0}=\const_*
    \text{ to }[\gamma_1,\const_*]^{\downarrow 1}=[\gamma_1,\const_*],\\
    [\const_*,\gamma_2]^{\downarrow \bull}\colon[0,1]\to\Omega(Z_1\vee Z_2)
    \text{ from }[\const_*,\gamma_2]^{\downarrow 0}=\const_*
    \text{ to }[\const_*,\gamma_2]^{\downarrow 1}=[\const_*,\gamma_2],
    \end{align*}
    by cutting the path short: 
    \begin{equation}\label{eq:nullhtpy}
        [\gamma_1,\const_*]^{\downarrow t}
        =\left\{\begin{array}{{@{}lc@{}}}
            \gamma_1(4\theta t)     & \theta\in[0,\tfrac{1}{4}]            \\[3pt]
            \gamma_1(t)             & \theta\in[\tfrac{1}{4},\tfrac{1}{2}]           \\[3pt]
            \gamma_1((3-4\theta)t)  & \theta\in[\tfrac{1}{2},\tfrac{3}{4}]  \\[3pt]
            \gamma_1(0)=*             & \theta\in[\tfrac{3}{4},1]
        \end{array}
            \quad\quad
            \begin{array}{{@{}l@{}}}
                 \gamma_2(0)=*              \\[3pt]
                 \gamma_2((4\theta-1)t)   \\[3pt]
                 \gamma_2(t)              \\[3pt]
                 \gamma_2((4\theta-3)t)
            \end{array}\right\}=[\const_*,\gamma_2]^{\downarrow t}.\qedhere
    \end{equation}
\end{defn}
This is exactly a situation from Lemma~\ref{lem:lift} about fat-smash products: even though we do not have a map out of $\Omega Z_1\wedge\Omega Z_2$, we can define the \emph{commutator bracket} as
\begin{equation}\label{eq:commutator}
    [-,-]\colon\;
    \Omega Z_1 \mywedge \Omega  Z_2
    \ra \Omega(Z_1\vee Z_2).
\end{equation}
This sends points $(\gamma_1,\gamma_2)\in\Omega Z_1\tm\Omega Z_2$ to $\gamma_1\gamma_2\gamma_1^{-1}\gamma_2^{-1}$, and the path from the cone point to $(\gamma_1,*)$ to the path $[\gamma_1,\const_*]^{\downarrow \bull}$, and the path from the cone point to $(*,\gamma_2)$ to the path $[\const_*,\gamma_2]^{\downarrow \bull}$.

\subsubsection{Samelson brackets}
Recall that $w\in\LL(n)$ is a Lie word in $n$ letters $x^1,\dots,x^n$.

\begin{defn}\label{def:Samelson}
    Fix collections of spaces $\{X_i\}_{i\in\ul{n}}$  and $\{Y_i\}_{i\in\ul{n}}$, and maps $\{f_i\colon X_i\to\Omega Y_i\}_{i\in\ul{n}}$, and denote $Y_S\coloneqq \bigvee_{i\in S} Y_i$.
    For $w\in B\LL(n)$ inductively define the space $w(\{X_i\})$ and the \emph{iterated Samelson bracket} $w(\{f_i\})\colon w(\{X_i\})\to\Omega Y_{\ul{n}}$, as follows.
    
    If $w=x^j$ then $w(\{x_i\})\coloneqq f_j$.
    Otherwise, we have $w=[w_1,w_2]$ for some $w_b\in B\LL(S_b)$ with $S_1\cup S_2=\ul{n}$ (note that potentially $S_1\cap S_2\neq\emptyset$). Then the maps $w_b(\{f_i\})\colon w_b(\{X_i\})\to\Omega Y_{S_b}$ have been defined by the inductive hypothesis. We define $w(\{X_i\})\coloneqq w_1(\{X_i\})\mywedge w_2(\{X_i\})$ and the map $w(\{f_i\})$ by
\[\begin{gathered}[b]
\begin{tikzcd}
     &[-1.5cm] \Omega Y_{S_1}\mywedge \Omega Y_{S_2}
        \rar{[-,-]} 
    & \Omega( Y_{S_1}\vee  Y_{S_2})
        \dar{\fold}\\
    w(\{X_i\})\coloneqq &[-2.5cm]
    w_1(\{X_i\})\mywedge w_2(\{X_i\})
        \uar{w_1(\{f_i\})\mywedge w_2(\{f_i\})} 
        \rar[densely dashed]{w(\{f_i\})}
    & \Omega Y_{\ul{n}}
\end{tikzcd}\\[-\dp\strutbox]
\end{gathered}\qedhere
\]
\end{defn}
Note that $w(\{X_i\})\simeq \bigwedge_{i\in\ul{n}} X_i^{\wedge a_i}$, where $a_i$ is the number of occurrences of the letter $x^i$ in $w$.

\subsubsection{Canonical Samelson maps}

In particular, taking for each $i\in\ul{n}$ the space $Y_i\coloneqq\Sigma X_i$ and the map $f_i\coloneqq x_i\colon X_i\to\Omega\Sigma X_i$ adjoint to the identity, we have the \emph{canonical Samelson maps}
\begin{equation}\label{eq:can-Sam}
    w(\{x_i\})\colon w(\{X_i\})\ra\Omega\Sigma X_{\ul{n}}.
\end{equation}
We can extend each $w(\{x_i\})$ multiplicatively to a map $hm_w\colon\Omega\Sigma w(\{X_i\})\to\Omega\Sigma X_{\ul{n}}$ given by $hm_w(\theta\mapsto t_\theta\wedge z)=(\theta\mapsto w(\{x_i\})(z)(t_\theta))$.
The weak product (filtered colimit of finite products) of these maps over all basic words $w$ in $B\LL(n)\subset\LL(n)$ (see §\ref{subsec:Lie}) gives the classical \emph{Hilton--Milnor equivalence}:
\begin{equation}\label{eq:HM-all}
\begin{tikzcd}
        hm_B=\prod hm_w\colon\;\prod_{w\in B\LL(n)}\Omega\Sigma w(\{X_i\})\rar{\sim} 
            &   \Omega\Sigma X_{\ul{n}}.
\end{tikzcd}
\end{equation}

%%%%%%%%%%%%%%%%%%%%%%%%%%%%%%%%%%%%%%%%%%%%%%%%%%%%%%%%%%%%%%%%%%
\section{Cubes and external Samelson brackets}
\label{sec:cubes-ext}

 Both homotopy and isotopy Taylor towers are defined using cubical diagrams, that offer a convenient language for multi-relative homotopy theory; we refer to \cite{Munson-Volic}. In §\ref{subsec:cubes} we recall the definition of cubes and total homotopy fibres, and including and collapsing cubes.
In §\ref{subsec:ext-Sam-brackets} we introduce the notion of external Samelson brackets: these generalise Samelson brackets to total homotopy fibres. In §\ref{subsec:examples} we study some examples. In §\ref{subsec:total-Sam} we extend the definition of Samelson maps~\eqref{eq:can-Sam} from loop spaces $\Omega\Sigma X_i$ to loop spaces of total homotopy fibres $\Omega\tofib(\Sigma X_{\ul{n}};\kappa)$.

%%%%%%%%%%%
\subsection{Cubes}
\label{subsec:cubes}

    Let $\mathcal{P}(n)$ be the poset of all subsets of $\ul{n}\coloneqq\{1,\dots,n\}$. A \emph{cube} (more precisely, a cubical diagram) is a functor $Y_{\bull}\colon \mathcal{P}(n)\to \Top_*$. For example, for $I\coloneqq[0,1]$ we have
    \[
    I^{\bull}\coloneqq\begin{cases}
        S &\mapsto\quad I^S\coloneqq\Map(S,I),\\
        S \subseteq T & \mapsto\quad \iota_S^T\colon I^S=I^S\tm\{1\}
        ^{T\sm S}\hra I^T.
    \end{cases}
    \]
    In fact, the cube $I^{\bull}\cong|\mathcal{P}(n)\downarrow\bull|$ is the geometric realisation of the cubical diagram of categories given by $S\mapsto\mathcal{P}(n)\downarrow S$, the category over $S\in\mathcal{P}$.

    For a cube $Y_{\bull}$ consisting of based spaces $(Y_S,*_S)$ and maps $y_S^T$ we will often write $(Y_{\bull};y)$ to make the dependence on maps clear. Let us define the \emph{total homotopy fibre} of a cube. For $t\in\{0,1\}$ we say that $p\in I^S$ belongs to a $t$-valued face if it has a coordinate equal to $t$. Let $A_S\subseteq I^S$ be the union of 0-valued faces. Then
    \[
        \tofib(Y_{\bull};y)\subseteq \prod_{S\subseteq\ul{n}}\Map((I^S,A_S),(Y_S,{*_S}))
    \]
    consists of collections of maps $\{F^S\colon (I^S,A_S)\to (Y_S,{*_S})\}$ that are mutually compatible on 1-valued faces, that is, for any $S\subseteq T$ the following diagram commutes:
    \begin{equation}\label{eq:tofib-pt}
    \begin{tikzcd}
        I^S\arrow[hook]{r}{\iota_S^T}\dar[swap]{F^S} & I^T\dar{\;F^T}\\
       Y_S\rar{y_S^T} & Y_T
    \end{tikzcd}
    \end{equation}
\begin{rem}\label{rem:cubes}
\begin{enumerate}
    \item
    Sometimes the opposite convention is used, with the role of 0- and 1-valued faces swapped, as in~\cite{K-thesis-paper}.
    % \item 
    % If $y_S^T$ are inclusions, and a map $F\colon I^{\ul{n}}\to Y_{\ul{n}}$ has the property that its restriction to each face $I^S\subseteq I^{\ul{n}}$ has image in the subspace $Y_S\subseteq Y_{\ul{n}}$, then $F$ \emph{canonically extends} to a point in $\tofib Y_{\bull}$. 
    \item 
     Below we will use the fact that total homotopy fibres commute with homotopy limits, and in particular loop spaces: $\Omega\tofib(Y_{\bull};y)\cong\tofib(\Omega Y_{\bull};\Omega y)$.
     \item 
     The symmetric group $\Sigma_n$ acts on the category $\mathcal{P}(n)$, so precomposing $(Y_\bull;y)\colon\mathcal{P}(n)\to\Top_*$ by $\sigma\in\Sigma_n$ gives another cube $(Y_{\sigma\bull};y)$. Moreover, we have a map $\tofib(Y_\bull;y)\to\tofib( Y_{\sigma\bull}; y)$ that sends $\{F^S\colon (I^S,A_S)\to (Y_S,{*_S})\}$ to $\{G^S\colon (I^S,A_S)\to (Y_{\sigma S},{*_{\sigma S}})\}$, with $G^S=F^{\sigma S}\circ \sigma^{-1}$.
     \qedhere
\end{enumerate}
\end{rem}

%%%%%%%%%%%

Let us now fix $n\geq1$ and based spaces $Y_0,Y_1,\dots,Y_n$ and for any $P\subseteq[n]\coloneqq\{0,1,\dots,n\}$ denote $Y_P\coloneqq \bigvee_{i\in P}Y_i$.
We consider the \emph{including cube}
\begin{equation}\label{eq:including-cube}
(Y_{0\ul{n}},\iota)\coloneqq
\begin{cases}
    S\subseteq\ul{n}    &\mapsto\quad Y_{0S},\\
    S \subseteq T       & \mapsto\quad \iota_S^T\colon Y_{0S}\hra Y_{0T},
\end{cases}
\end{equation}
and also the contravariant \emph{collapsing cube} (note the contravariance!)
\begin{equation}\label{eq:collapsing-cube}
(Y_{0\ul{n}};\kappa)\coloneqq
\begin{cases}
    S\subseteq\ul{n}    &\mapsto\quad Y_{0S},\\
    S \subseteq T       & \mapsto\quad \kappa_S^T\colon Y_{0T}\sra Y_{0S}.
\end{cases}
\end{equation}
Often we will have $Y_0=*$ so we omit the index $0$; the case where $Y_0\neq*$ is considered in §\ref{subsec:J-general}.
Let $\proj\colon\tofib(Y_{0\ul{n}};\kappa)\to Y_{0\ul{n}}$ be the canonical projection:  $\proj(\{F^S\})=F^\emptyset\colon I^0\to Y_{0\ul{n}}$. 

We have $\kappa_S^T\circ\iota_S^T=\Id$, that is, $\kappa_S^T$ is a \emph{left} homotopy inverse of $\iota_S^T$. This implies that $\hofib(\iota_S^T)\simeq\Omega\hofib(\kappa_S^T)$. Generalising this equivalence to total homotopy fibres (which can be described as iterated homotopy fibres), yields the following;  see \cite[Appendix~A]{K-thesis-paper} for a proof.
\begin{lem}\label{lem:deloop}
    The composition of the projection map and the inclusion of cubes constant on all faces:
    \[
    \chi^{-1}\colon\Omega^n\tofib\left(Y_{0\ul{n}}; \kappa\right)\ra 
    \Omega^nY_{0\ul{n}}\ra
    \tofib\left(Y_{0\ul{n}}; \iota\right)
    \]
    is a homotopy equivalence. 
    Moreover, the canonical forgetful map induces an injection
    \[
    \pi_*\tofib(Y_{0\ul{n}}; \kappa)\hra \pi_*(Y_{0\ul{n}}; \kappa).
    \]
\end{lem}

Note that $\chi^{-1}$ sends $F=\{F^S\colon I^{\ul{n}\sm S}\to\Omega^nY_{0S}\}_{S\subseteq\ul{n}}$ to the collection $\chi^{-1}(F)=\{F^{\ul{n}}\}\cup\{\const_S\}_{S\subset\ul{n}}$ where $F^{\ul{n}}\colon I^\emptyset\tm I^n\to Y_{0\ul{n}}$ maps the boundary to the basepoint. 
In fact, in \cite[Appendix~A]{K-thesis-paper} we give an explicit homotopy equivalence $\chi\colon \tofib\left(Y_{0\ul{n}}; \iota\right) \to \Omega^n\tofib\left(Y_{0\ul{n}}; \kappa\right)$, that will be used in §\ref{subsec:istpy-calc}.

%%%%%%%%%%%
\subsection{Definition of external Samelson brackets}
\label{subsec:ext-Sam-brackets}

Fix a collection $\{Y_i\}_{i\in\ul{n}}$.
Recall that for a finite set $S$ we write $Y_S\coloneqq\bigvee_{i\in S}Y_i$, and that $(Y_S;\kappa)$ is the collapsing $|S|$-cube as in~\eqref{eq:collapsing-cube}.

\subsubsection{Disjoint indexing sets}
We first fix two subsets $S_b\subseteq\ul{n}$ with $S_1\sqcup S_2=\ul{n}$, so that $Y_{S_1}\vee Y_{S_2}=Y_{\ul{n}}$, and we construct a map
\[
   \tofib( \Omega Y_{S_1}; \Omega\kappa)\mywedge 
   \tofib( \Omega Y_{S_2}; \Omega\kappa) \ra 
   \tofib( \Omega Y_{\ul{n}}; \Omega\kappa). 
\]
First, given points $F_b\in\tofib(\Omega Y_{S_b};\kappa)$ for $b=1,2$, described by collections
\[
    \Big\{ F_b^{T_b}\colon (I^{S_b\sm T_b},A_{S_b\sm T_b}) \ra 
    (\Omega Y_{T_b},\const_{*_{T_b}})\Big\}_{T_b\subseteq S_b}
\]
our goal is to produce a point $[F_1,F_2]_{\ext}$ in the total homotopy fibre of the $n$-cube $(\Omega Y_{\ul{n}}; \Omega\kappa)$ described by a collection
\begin{equation}\label{eq:comm}
    \Big\{ [F_1^{T_1}, F_2^{T_2}]\colon
    (I^{\ul{n}\sm T},A_{\ul{n}\sm T}) \ra 
    (\Omega Y_{\ul{n}},\const_*)\Big\}_{T\subseteq \ul{n}}.
\end{equation}
To this end, we can use Lemma~\ref{lem:naturality} to form the fat-smash product of two maps of pairs, and then take the commutator map~\eqref{eq:commutator}, to obtain:
\[\begin{tikzcd}[column sep=small]
    \big(I^{S_1\sm T_1},A_{S_1\sm T_1}\big) \mytm 
    \big(I^{S_2\sm T_2},A_{S_2\sm T_2}\big)
    \ar{rr}{F_1^{T_1}\mytm F_2^{T_2}} &&
    \big(\Omega Y_{T_1},\const_{*_{T_1}}\big) \mytm 
    \big(\Omega Y_{T_2},\const_{*_{T_2}}\big) \rar{[-,-]} & 
    \big(\Omega(Y_{T_1}\vee Y_{T_2}),\const_*\big).
\end{tikzcd}
\]
Observe that $I^{S_1\sm T_1}\tm I^{S_2\sm T_2}=I^{\ul{n}\sm T}$ for the set $T\coloneqq T_1\sqcup T_2$, and $A_{S_1\sm T_1}\tm I^{S_2\sm T_2}\cup_{A_{S_1\sm T_1}\tm A_{S_2\sm T_2}}I^{S_1\sm T_1}\tm A_{S_2\sm T_2}$ is precisely the union of all 0-valued faces $A_{\ul{n}\sm T}\subset I^{\ul{n}\sm T}$. Therefore,
\begin{equation}\label{eq:Phi}
    \begin{tikzcd}
    \big(I^{S_1\sm T_1},A_{S_1\sm T_1}\big)
    \mytm 
    \big(I^{S_2\sm T_2},A_{S_2\sm T_2}\big)
    =
    \Big(A_{\ul{n}\sm T}\tm[0,1]\cup_a I^{\ul{n}\sm T},\; A_{\ul{n}\sm T}\tm\{0\}\Big)\rar{\Phi}[swap]{\cong} &
    \big(I^{\ul{n}\sm T},A_{\ul{n}\sm T}\big),
\end{tikzcd}
\end{equation}
where $\Phi$ is a homeomorphism that sends $I^{\ul{n}\sm T}$ onto $[1/3,1]^{\ul{n}\sm T}$ (that will be called the ``inner cube'') and $A_{\ul{n}\sm T}\tm[0,1]$ onto the rest of $[0,1]^{\ul{n}\sm T}$ (called the ``$A$-collar''). Whence, we indeed have maps~\eqref{eq:comm}.

Second, we check that these maps assemble into a point $[F_1,F_2]_{\ext}\in\tofib(\Omega Y_{\ul{n}};\Omega\kappa)$. That is, for every $T\subseteq T'$ we need to verify that
\[
    [F_1^{T_1}, F_2^{T_2}]\circ \iota^{S\sm T}_{S'\sm T'}=\Omega\kappa^{S\sm T}_{S'\sm T'}\circ [F_1^{T'_1}, F_2^{T'_2}].
\]
This says that the restriction of $[F_1^{T_1}, F_2^{T_2}]$ onto the face where $x_i=1$ for all $i\in T'\sm T$ is obtained from $[F_1^{T'_1}, F_2^{T'_2}]$ by collapsing $Y_i$ with $i\in T'\sm T$. This is clear on the ``inner cube'' part of 1-valued faces where we have commutator loops of $F_b^{T_b}$, that do restrict to $F_b^{T'_b}$ on 1-valued faces by hypothesis. On the part of 1-valued faces that belongs to the ``$A$-collar'' we perform canonical homotopies along the lines, so the previous case implies the compatibility in this case.

Third, if one of the collections $\{F_b^{T_b}\}$ is trivial then the whole $[F_1,F_2]$ admits a canonical nullhomotopy itself, that is, a path to the basepoint of $\tofib( \Omega Y_{\ul{n}};\Omega\kappa)$, given as the collection of constant maps, again using formulae~\eqref{eq:nullhtpy}. Thus, the displayed map descends to the fat-smash product of total homotopy fibres.

\subsubsection{Overlapping indexing sets}

Assume now that $S_1\cup S_2=\ul{n}$ and $S_1\cap S_2$ is possibly nonempty. We can still apply the previous construction to obtain a map
\[
   \tofib( \Omega Y_{S_1};\Omega\kappa)\mywedge \tofib(\Omega Y_{S_2};\Omega\kappa)\ra \tofib( \Omega (Y_{S_1}\vee Y_{S_2});\Omega\kappa), 
\]
where on the right we have an $(|S_1|+|S_2|)$-cube.

We can then apply the map $\fold\colon Y_{S_1}\vee Y_{S_2}\to Y_{\ul{n}}$ that folds the repeating factors: for every $i\in S_1\cap S_2$ use the map $Y_i\vee Y_i\to Y_i$ that is the identity on each factor. This collapses the total homotopy fibre of the above $(|S_1|+|S_2|)$-cube to the total homotopy fibre of the $n$-cube $(\Omega Y_{\ul{n}};\Omega\kappa)$. This simply forgets higher homotopies involving the same index.
 
\begin{defn}\label{def:external-Samelson}
    We define the \emph{external Samelson bracket} $[-,-]_{\ext}$ by the diagram
    \[
    \begin{gathered}[b]
    \begin{tikzcd}
        \tofib(\Omega Y_{S_1};\Omega\kappa)\mywedge \tofib(\Omega Y_{S_2};\Omega\kappa)
            \rar{} 
        & \tofib(\Omega(Y_{S_1}\vee Y_{S_2});\Omega\kappa)
            \dar{\fold}\\
        \Omega\tofib(Y_{S_1};\kappa)\mywedge \Omega\tofib(Y_{S_2};\kappa)
            \uar[equals]
            \rar[densely dashed]{[-,-]_{\ext}}
        & \Omega\tofib(Y_{\ul{n}};\kappa).
    \end{tikzcd}\\[-\dp\strutbox]
    \end{gathered}\qedhere
    \]
\end{defn}

%%%%%%%%%%%
\subsection{Some examples}
\label{subsec:examples}

In this section we work out some key examples.

\begin{ex}\label{ex:cube-deg-2}
    For $n=2$ we start with $F_b\in\tofib(\Omega Y_b\to *;\Omega\kappa)$ for $b=1,2$. Here the total homotopy fibre is just the homotopy fibre, so each $F_b$ consists of $F_b^{\{b\}}\in \Omega Y_b$ and $F_b^{\emptyset}=\const_*\colon I\to \Omega *$. The point $G\coloneqq[F_1,F_2]_{\ext}\in\tofib(\Omega(Y_1\vee Y_2);\kappa)$ consists of the following data.
\begin{itemize}
    \item The point $G^{\{1,2\}}\in \Omega(Y_1\vee Y_2)$ is the commutator of the loops $F_1^{\{1\}}$ and $F_2^{\{2\}}$.
    \item The path $G^{\{2\}}\colon I\to \Omega Y_2$ is the concatenation of the path on $[1/3,1]$ constantly equal to $[*,F_2^{\{2\}}]$ and the path sending $t_1\in[0,1/3]$ to $[*,F_2^{\{2\}}]^{\downarrow 3t_1}$. Namely, in this case $A=\{0\}$ and the homeomorphism $\Phi^{-1}$ from~\eqref{eq:Phi} sends $t_1\in[0,1/3]$ in the $A$-collar to $3t_1\in\{0\}\tm[0,1]$, to which we apply~\eqref{eq:nullhtpy}. 
    \item The path $G^{\{1\}}\colon I\to \Omega Y_1$ is analogous to the previous.
    \item The 2-parameter family $G^\emptyset=\const_*\colon I^2\to *$ is trivial.
\end{itemize}
    We can rewrite $\Omega\tofib(Y_{\{12\}};\kappa)
    \cong\Omega\hofib(Y_1\vee Y_2\to Y_1\tm Y_2)$, and in the paths $G^{\{b\}}$ we can mod out the constant parts on $[1/3,1]$. Then $G\coloneqq[F_1,F_2]_{\ext}$ is described by the loop $G^{\{1,2\}}=[F_1^{\{1\}},F_2^{\{2\}}]\in \Omega(Y_1\vee Y_2)$ together with two paths $G^{\{b\}}\colon I\to\Omega Y_b$ given by $G^{\{1\}}(t_2)=[F_1^{\{1\}},*]^{\downarrow t_2}$ and $G^{\{2\}}(t_1)=[*,F_2^{\{2\}}]^{\downarrow t_1}$. 
\end{ex}

\begin{ex}\label{ex:cube-deg-2-compute}
    For $t\in I$ (time parameter of a path) and $\theta\in I$ (loop parameter) consider 
    \[
    G^{\{1\}}(t)(\theta)\wedge G^{\{2\}}(t)(\theta)\coloneqq[F_1^{\{1\}},*]^{\downarrow t}(\theta)\wedge [*,F_2^{\{2\}}]^{\downarrow t}(\theta)\in Y_1\wedge Y_2,
    \]
    the pointwise smash product.
    Using~\eqref{eq:nullhtpy} this can be spelled out as the left-hand side below, that is clearly homotopic to the right-hand side:
    \[
       \left(\begin{array}{{@{}cc@{}}}
            F_1^{\{1\}}(4\theta t) \wedge *                  & \theta\in[0,\tfrac{1}{4}]\\[3pt]
            F_1^{\{1\}}(t)\wedge F_2^{\{2\}}((4\theta-1)t)         & \theta\in[\tfrac{1}{4},\tfrac{1}{2}]\\[3pt]
            F_1^{\{1\}}((3-4\theta)t) \wedge F_2^{\{2\}}(t)       & \theta\in[\tfrac{1}{2},\tfrac{3}{4}]\\[3pt]
            *\wedge F_2^{\{2\}}((4\theta-3)t)                 & \theta\in[\tfrac{3}{4},1]
        \end{array}\right)
        \simeq
        \left(\begin{array}{{@{}cc@{}}}
            F_1^{\{1\}}(t)\wedge F_2^{\{2\}}(2\theta t)           & \theta\in[0,\tfrac{1}{2}]\\[3pt]
            F_1^{\{1\}}((2-2\theta)t) \wedge F_2^{\{2\}}(t)       & \theta\in[\tfrac{1}{2},1]
        \end{array}\right).
    \]
   Note that this is defined on $(t,\theta)\in I^2$ and sends the boundary of $I^2$ to $*$, so defines a map $\S^1\wedge\S^1\to Y_1\wedge Y_2$. Moreover, by Lemma~\ref{lem:square-htpy} this is homotopic to $\theta\wedge t \mapsto F_1^{\{1\}}(t)\wedge F_2^{\{2\}}(\theta)$.
\end{ex}
 The following two results give an appropriate generalisation for $n\geq2$ of the last two examples; for applications in the proof of Theorem~\ref{thm:J} we will need an auxiliary parameter $\tau$.
\begin{ex}\label{ex:F-i}
    For $n\geq2$ and $S_1\sqcup S_2=\ul{n}$ fix $i\in S_1$. Let $G\coloneqq[F_1,F_2]_{\ext}\in\tofib(\Omega(Y_{S_1}\vee Y_{S_2});\kappa)$.
    Then
    \[
    G^{\{i\}}\colon I^{S_1\sm\{i\}}\tm I^{S_2}\to \Omega (Y_i\vee *)
    \ra \Omega Y_i
    \]
    sends $(t_j)_{j\in S_1\sm\{i\}}\tm(t_k)_{k\in S_2}$ to the commutator $[F_1^{\{i\}}((\tfrac{t_j-1/3}{1-1/3}))
    ,\const_*]$ if $\tfrac{1}{3}\leq t_k\leq 1$ for all $k\in S_2$, and to the canonical nullhomotopy 
    $(t_j)_{j\in S_1\sm\{i\}}\tm(t_k)_{k\in S_2}\mapsto [F_1^{\{i\}}((\tfrac{t_j-1/3}{1-1/3}))
    ,\const_*]^{\downarrow 3\min_{k\in S_2}t_k}
    $ when at least one $0\leq t_k\leq \tfrac{1}{3}$. The formula again originates in our choice for the homeomorphism $\Phi^{-1}$ in~\eqref{eq:Phi}. Note that when some $t_j<\tfrac{1}{3}$ we use the trivial nullhomotopy for $[\const_*,\const_*]$. 
\end{ex}

\begin{lem}\label{lem:F-i}
    Let $n\geq2$ and $S_1\sqcup S_2=\ul{n}$. Let $F_b\in\Omega \tofib(Y_{S_b};\kappa)$ for $b=1,2$ and $G\coloneqq[F_1,F_2]_{ext}$. Fix $\tau\in(0,1)$ and $i\in\ul{n}$. Then for the coordinate $G^{\{i\}}\colon I^{\ul{n}\sm\{i\}}\to Y_i$
there is a homotopy preserving boundary strata to the map
    \[\begin{cases} \text{ if }i\in S_1:
    & (t_j)_{j\in S_1\sm\{i\}}\tm (t_k)_{k\in S_2}
    \mapsto
        [F_1^{\{i\}}((\tfrac{t_j-\tau}{1-\tau})),\const_*]^{\downarrow \min_{k\in S_2}t_k},
\\
    \text{ if }i\in S_2:   
    & (t_j)_{j\in S_1}\tm (t_k)_{k\in S_2\sm\{i\}}
    \mapsto 
        [\const_*,F_2^{\{i\}}((\tfrac{t_k-\tau}{1-\tau}))]^{\downarrow \min_{j\in S_1}t_j},
    \end{cases}
    \]
    with the convention that for $t_m<\tau$ we set $\tfrac{t_m-\tau}{1-\tau}=0$.
\end{lem}
\begin{proof}
    In the definition~\eqref{eq:Phi} of $\Phi$ we can instead of $[1/3,1]^{\ul{n}}$ use $[\tau,1]^{\ul{n}}$ as the inner cube. More precisely, there is an isotopy through homeomorphisms $\Phi_\tau\colon(A_{\ul{n}\sm T}\tm[0,1]\cup_a I^{\ul{n}\sm T},\; A_{\ul{n}\sm T}\tm\{0\})\to (I^{\ul{n}\sm T},A_{\ul{n}\sm T})$,  $\tau\in(0,1)$, such that $\Phi_\tau(I^{\ul{n}\sm T})=[\tau,1]^{\ul{n}\sm T}$ and so that $\Phi=\Phi_{1/3}$.

    Therefore, if $i\in S_1$, say, the map $G^{\{i\}}$ becomes $[F_1^{\{i\}}((\tfrac{t_j-\tau}{1-\tau})),\const_*]$ when $\tau\leq t_k\leq 1$ for all $k\in S_2$, and $[F_1^{\{i\}}((\tfrac{t_j-\tau}{1-\tau})),\const_*]^{\downarrow \tfrac{1}{\tau}\min_{k\in S_2}t_k}$ otherwise.
    Since this is constant in all $t_k\in[\tau,1]$, we can collapse these edges onto the vertex $1$, giving the desired formula with $\min_{k\in S_2}t_k$ in the exponent.
\end{proof}

%%%%%%%%%%%
\subsection{Total Samelson maps}
\label{subsec:total-Sam}

If we assume $Y_i=\Sigma X_i$ is a suspension for each $i=1,\dots,n$, then we can use the Hilton--Milnor equivalence~\eqref{eq:HM-all} to describe the homotopy type of the collapsing cube. Namely, if in $w\in B\LL(n)$ each letter appears at least once, then the canonical Samelson map $w(\{x_i\})\colon w(\{X_i\})\to\Omega\Sigma X_{\ul{n}}$ from \eqref{eq:can-Sam} extends to a map $w(\{x_i\})^\bull$ to $\Omega\tofib\left(\Sigma X_{\ul{n}}; \kappa\right)$, using the external bracket as follows.

% Namely, for such a word $w$ if the map $x_i$ is replaced by $\const_*$, then we can use the canonical nullhomotopy as in \eqref{eq:nullhtpy} to construct a homotopy $w(\{x_i\})^{\ul{n}\sm\{i\}}\colon w(\{X_i\})\tm  I\to \Omega\Sigma X_{\ul{n}\sm\{i\}}$ from $\const_*$ to $w( x_k,*_i)\colon w(\{X_i\})\to\Omega\Sigma X_{\ul{n}\sm\{i\}}$.
% If two maps $x_i,x_k$ are replaced by $\const_*$ these paths can be related by a family $w(\{x_i\})^{\{i,k\}}\colon w(\{X_i\})\tm I^2\to \Omega\Sigma X_{\ul{n}\sm\{i,k\}}$, and so on.
% In fact, this is equivalent to the following construction that uses the external Samelson bracket from Definition~\ref{def:external-Samelson}.
\begin{defn}
    For $w\in B\LL(n)$ we inductively define the \emph{total Samelson map}
\begin{equation}\label{eq:total-Samelson}
    w(\{x_i\})^\bull\colon w(\{X_i\})\ra\Omega\tofib\left(\Sigma X_{\ul{n}}; \kappa\right),
\end{equation}
    as the following collection of maps $w(\{x_i\})^S\colon w(\{X_i\})\tm I^{\ul{n}\sm S}\to \Omega\Sigma X_S$ for $S\subseteq\ul{n}$.

    For the base, if $w=x^j$ define $w(\{x_i\})^\bull\colon X_j\to\Omega\hofib(\Sigma X_j\to *)$ as the canonical map $w(\{x_i\})^{\{j\}}=x_j\colon X_j\to \Omega\Sigma X_j$, together with the trivial nullhomotopy $w(\{x_i\})^{\emptyset}=\const_*\colon X_j\tm I\to \Omega *$. 
    
    Now take $w=[w_1,w_2]\in B\LL(n)$ in which each letter $i\in\ul{n}$ appears, and $w_b\in\LL(S_b)$ for some $S_1\cup S_2=\ul{n}$. Then define $w(\{x_i\})^\bull$ by the commutative square:
    \[\begin{gathered}[b]
    \begin{tikzcd}
        \Omega \tofib(\Sigma X_{S_1};\kappa)\mywedge \Omega\tofib(\Sigma X_{S_2};\kappa)
            \rar{[-,-]_{\ext}}
        & \Omega \tofib(\Sigma X_{\ul{n}};\kappa)
        \\
        w_1(\{X_i\})
            \mywedge
        w_2(\{X_i\})
            \rar{\cong}
            \uar{w_1(\{x_i\})^\bull\mywedge w_2(\{x_i\})^\bull}
        & w(\{X_i\})
            \uar[densely dashed,swap]{w(\{x_i\})^\bull}.
    \end{tikzcd}\\[-\dp\strutbox]
    \end{gathered}\qedhere
    \]
\end{defn}
Recall that $\proj\colon\tofib(\Sigma X_{\ul{n}};\kappa)\to X_{\ul{n}}=\bigvee_{i\in\ul{n}}X_i$ is the canonical projection $\proj(\{F^S\})=F^\emptyset(I^0)$. The following is proven easily by induction.
\begin{lem}\label{lem:total-vs-canonical}
     We have $\proj\circ w(\{x_i\})^\bull=w(\{x_i\})$, that is, the canonical Samelson maps~\eqref{eq:can-Sam} are precisely projections of the total Samelson maps~\eqref{eq:total-Samelson} to the 1-valued vertex.
\end{lem}

In fact, $\pi_*\Omega\tofib\left(\Sigma X_{\ul{n}}; \kappa\right)$ is precisely the summand of $\pi_*(\Omega\Sigma X_{\ul{n}})$ which maps to zero by each of the collapse maps $\kappa$ (see Lemma~\ref{lem:deloop}).
Therefore, we have (see for example \cite{GooIII,GW}):
\begin{equation}\label{eq:HM}
        hm_{NB}=\prod hm_w\colon\;\prod_{w\in NB\LL(n)}\Omega\Sigma w(\{X_i\})\xrightarrow{\sim}
            \Omega\tofib\left(\Sigma X_{\ul{n}}; \kappa\right),
\end{equation}
where $NB\LL(n)\subset B\LL(n)$ is the subset of the Lyndon basis consisting of those basic words $w$ which include each letter at least once, and $hm_w$ is the multiplicative extension of $w(\{x_i\})^\bull$.

%%%%%%%%%%%
% \subsubsection{The Lie range}

Recall from \eqref{eq:right-normed} that $B\Lie(n)=(B\LL(n))\cap\Lie(n)=NB\LL(n)\cap\{\text{length }=n\}$ consists of basic words $w_\sigma\in NB\LL(n)$ of length exactly $n$, and is parameterised by $\sigma\in\Sigma_{n-1}$. Thus, restricting \eqref{eq:HM} to this subset and using the Freudenthal Suspension Theorem gives isomorphisms
\[\begin{tikzcd}
    \bigoplus_{\sigma\in \Sigma_{n-1}} \pi_mQ w_\sigma(\{X_i\})
    \cong \bigoplus_{\sigma\in \Sigma_{n-1}} \pi_m w_\sigma(\{X_i\})
    \cong  \pi_m \prod_{w\in B\Lie(n)} \Omega\Sigma w(\{X_i\})
    \rar{hm|}[swap]{\cong}
    & \pi_m \Omega\tofib\left(\Sigma X_{\ul{n}}; \kappa\right)
\end{tikzcd}
\]
for $0\leq m\leq (n+1)(c+1)-1$. Note that $\pi_mQZ$ is a stable homotopy group of a space $Z$.
\begin{cor}\label{cor:hm}
    If $X_i$ are $c$-connected for $c\geq0$, then $\prod w_\sigma(\{x_i\})^\bull\colon \prod_{\sigma\in\Sigma_{n-1}} w_\sigma(\{X_i\})\to
    \Omega\tofib(\Sigma X_{\ul{n}}; \kappa)$
    is $(n+1)(c+1)$-connected. Moreover, homotopy groups of these spaces are trivial for $0\leq m\leq n(c+1)-1$ and we have $\pi_{n(c+1)} \Omega\tofib\left(\Sigma X_{\ul{n}}; \kappa\right)
    \cong \bigoplus_{(n-1)!} \bigotimes_{i=1}^n H_{c+1}(X_i)$.
\end{cor}

\begin{ex}\label{ex:Lie}
    In particular, if $X_i=\S^{d-2}$ the first nonvanishing homotopy group is
    \[
        \pi_{n(d-2)}\Omega\tofib\left(\S^{d-1}_{\ul{n}}; \kappa\right)
        \cong\bigoplus_{(n-1)!}\Z\cong\Lie(n).\qedhere
    \]
\end{ex}

\section{Improvements}
\label{sec:improve}

The total homotopy fibre in Corollary~\ref{cor:hm} has a natural action by the symmetric group $\Sigma_n$ (permuting the $X_i$'s), but this is lost under the Hilton--Milnor equivalence -- the set of basic words is not closed under the action of $\Sigma_n$. 
Nevertheless, a suitable $\Sigma_n$-model was found by Johnson~\cite{Johnson}, reinterpreted by Arone--Kankaanrinta~\cite{AK-Snaith} and Arone--Mahowald~\cite{Arone-Mahowald}. 
In §\ref{subsec:J-collapsing} we give a new proof. 

This is based on the key Theorem~\ref{thm:J} in §\ref{subsec:J} that defines Johnson's map $J$ directly in terms of the Goodwillie partition complex, and then shows its compatibility with the grafting bracket from Definition~\ref{def:graft-bracket} and the external Samelson bracket from Definition~\ref{def:external-Samelson}.

Finally, in §\ref{subsec:J-general} we generalise these results to a setting where not all wedge factors are suspensions.

%%%%%%%%%%%
\subsection{The map J}
\label{subsec:J}

Let $F=\{F^S\colon I^{\ul{n}\sm S}\to Y_S\}$ be a point in $\tofib\left(Y_{\ul{n}}; \kappa\right)$ for some spaces $Y_i$, $1\leq i\leq n$. In particular, we have maps $F^{\{i\}}\colon I^{\ul{n}\sm\{i\}}\to Y_i$.
    For a weighted tree $(\Gamma,l)\in\GPC_{\ul{n}}$ we define
    \[
    J(F)(\Gamma,l)
    \coloneqq
        \bigwedge_{i\in\ul{n}} F^{\{i\}}((T_{ih})_{h\in\ul{n}\sm\{i\}})\quad\in \bigwedge_{i\in\ul{n}} Y_i
    \]
    where $T_{ih}=T_{hi}$ is the distance in $(\Gamma,l)$ to the root from the first common descendant $v_{ih}$ in $\Gamma$ of the leaves $i$ and $h$. Equivalently, $1-T_{ih}$ is the distance from the leaf $i$ to the vertex $v_{ih}$; it is also half the distance from $i$ to $h$.
    
Note that for $n=1$ we simply have $F=\{F^1\colon I^\emptyset\to Y_1\}$, and $\GPC_{\ul{1}}=\S^0$ has only one nontrivial tree $(|,1)$. Then $J(F)(*)=*\in Y_1$ and $J(F)(|,1)=F^1\in Y_1$.

\begin{thm}\label{thm:J}
    For any collection $\{Y_i\}_{i\in\ul{n}}$ this defines a $\Sigma_n$-equivariant map
    \[
        J\colon\tofib\left(Y_{\ul{n}}; \kappa\right)
        \ra \Map_*(\GPC_{\ul{n}}, \bigwedge_{i\in\ul{n}} Y_i),
    \]
    such that for two disjoint sets $S_1,S_2$ with $S_1\sqcup S_2=\ul{n}$ the following square commutes up to homotopy:
    \begin{equation}\label{eq:Lie-vs-ext}
    \begin{tikzcd}[column sep=large]
        \Omega \Map_*(\GPC_{S_1},\bigwedge_{i\in S_1}Y_i)\wedge \Omega \Map_*(\GPC_{S_2},\bigwedge_{i\in S_2}Y_i)
            \rar{[-,-]_{\graft}}
        & \Omega \Map_*(\GPC_{\ul{n}},\bigwedge_{i\in \ul{n}}Y_i)
        \\
        \Omega \tofib(Y_{S_1};\kappa)\mywedge \Omega \tofib(Y_{S_2};\kappa)
            \rar{[-,-]_{\ext}}
            \uar{\wt{q}\circ(\Omega J_{S_1}\mywedge \Omega J_{S_2})}
        & \Omega \tofib(Y_{\ul{n}};\kappa)
            \uar[swap]{\Omega J_{\ul{n}}}.
    \end{tikzcd}
    \end{equation}
\end{thm}
\begin{proof}
\textbf{$J$ is well defined.}
    We check that for $F\in\tofib\left(Y_{\ul{n}}; \kappa\right)$ the map $J(F)\colon \GPC_{\ul{n}}\to\bigwedge_{i\in \ul{n}}Y_i$ 
    is well defined. If the root edge has length $T_0=l(e_0)$ equal to zero, then some $T_{ih}=0$ and $F^{\{i\}}=*$ on that 0-valued face, whereas if from a leaf $i$ grows an edge of zero length, then there is a leaf $h$ for which $T_{ih}=1$, and on this 1-valued face 
    the map $F^{\{i\}}\wedge F^{\{h\}}$ is $*$, as here $F^{\{i\}}\tm F^{\{h\}}$ factors through the map $F^{\{ih\}}$ into the wedge.

\textbf{$J$ is equivariant.}
    We next check that the map $J\colon\tofib(Y_{\ul{n}}; \kappa)\to\Map_*(\GPC_{\ul{n}}, \bigwedge_{i\in\ul{n}} Y_i)$ is $\Sigma_n$-equivariant, by which we mean that the following square commutes:
    \[\begin{tikzcd}[column sep=small, row sep =small]
        \tofib(Y_{\ul{n}};\kappa)\rar{J}\dar{\sigma_*}
        & \Map_*(\GPC_{\ul{n}},\bigwedge_{i\in \ul{n}}Y_i) \dar{\sigma_*}
        \\
        \tofib(Y_{\sigma\ul{n}};\kappa)\rar{J}
        & \Map_*(\GPC_{\ul{n}},\bigwedge_{i\in \ul{n}}Y_i).
    \end{tikzcd}
    \]
    Here the vertical arrows on the left is the action from Remark~\ref{rem:cubes}(iii), and the one on the right uses $\Sigma_n$-action on $\ul{n}$ and the usual action on mapping spaces: $(\sigma_* f) (\Gamma,l)= \sigma f(\sigma^{-1} (\Gamma,l))$. Let $f=J(F)$ and $\sigma$ be the transposition $(j,k)$; the general case will follow. The cube $Y_{\sigma\bull}$ has the vertices $Y_j$ and $Y_k$ exchanged, and $\sigma^{-1} (\Gamma,l)$ has the leaves $j$ and $k$ exchanged. Then to obtain $J(\sigma_* F)(\Gamma,l)$ from $J(F)(\sigma^{-1}(\Gamma,l))$ we exchange the value $F^{\{j\}}((T_{jh})_{h\in\ul{n}\sm\{j\}})$ with $F^{\{k\}}((T_{kh})_{h\in\ul{n}\sm\{k\}})$ in $\bigwedge_{i\in\ul{n}} Y_i$, so indeed
    \[
     J(\sigma_* F)(\Gamma,l)=\sigma\big(
     J(F)(\sigma^{-1}(\Gamma,l))
     \big).
    \]
    
\textbf{Up-Right.}
    Consider the composite $UR$ of going up and to the right in the square~\eqref{eq:Lie-vs-ext}. As the map $\wt{q}$ quotients out the ``nullhomotopies of the wedge'' (see the sentence after Definition~\ref{def:my-smash}), $UR$ factors through the smash product:
\[
    UR\colon\; \Omega \tofib(Y_{S_1};\kappa)\mywedge \Omega \tofib(Y_{S_2};\kappa)\ra 
    \Omega \tofib(Y_{S_1};\kappa)\wedge \Omega \tofib(Y_{S_2};\kappa)\ra
    \Omega \Map_*(\GPC_{\ul{n}},\textstyle\bigwedge_{i\in \ul{n}}Y_i).
\]
    For $F_b=\{F_b^{T_b}\}_{T_b\subseteq S_b}\in \Omega\tofib(Y_{S_b};\kappa)$ for $b=1,2$, we have $UR(F_1,F_2)=[\Omega J_{S_1}(F_1),\Omega J_{S_2}(F_2)]_{\graft}$ by definition. Let us unpack this.
    
    Recall from Definition~\ref{def:graft-bracket} that the grafting bracket $[-,-]_{\graft}$ uses the ungrafting map $\ungraft_{S_1,S_2}$ from~\eqref{eq:cooperad}, that sends a weighted tree $(\Gamma,l)$ with $n$ leaves to the basepoint unless $T_0=l(e_0)\in(0,1)$ and $\Gamma$ is of the shape $\Gamma=[\Gamma_1,\Gamma_2]$ for trees with leaf sets $S_1$ and $S_2$ in some order; in the latter cases ungrafting outputs the tuple of $T_0=l(e_0)$ and two weighted trees $(\Gamma_b,l_b)$, with lengths of edges $l_b(e)=\tfrac{l(e)}{1-T_0}$. Thus, in the latter case the length in the ungrafted tree from the root to $v_{ih}$ for $i,h\in S_b$ is $\tfrac{T_{ih}-T_0}{1-T_0}$, so by the definition of $J$ we have
    \[
        \Omega J_{S_b}(F_b)(\Gamma_b,l_b)=
        \Big(
            \theta\mapsto\bigwedge_{i\in S_b} F_b^{\{i\}}
            \big(
                (\tfrac{T_{ih}-T_0}{1-T_0})_{h\in S_b\sm\{i\}}
            \big)
            (\theta)
        \Big).
    \]
    Thus, we can use Lemma~\ref{lem:square-htpy} to write $UR(F_1,F_2)$ as a map $\S^1\wedge\GPC_{\ul{n}}\to \bigwedge_iY_i$ that is given by 
    \begin{align}\label{eq:lhs}
        &\theta\wedge (\Gamma,l) \mapsto \\
        &\begin{cases}
            \bigwedge_{i\in S_1}
                F_1^{\{i\}}\big((\tfrac{T_{ih}-T_0}{1-T_0})_{h\in S_1\sm\{i\}}\big)
                (T_0)
            \wedge 
            \displaystyle\bigwedge_{i\in S_2}
                F_2^{\{i\}}\big((\tfrac{T_{ih}-T_0}{1-T_0})_{h\in S_2\sm\{i\}}\big)
                (\theta), & \text{if $T_0\in(0,1)$, leaf sets $S_1,S_2$,}\\
            *, & \text{otherwise.}
        \end{cases}\nonumber
    \end{align}
    
    % In addition, we claim that the formula~\eqref{eq:lhs} describes $UR(F_1,F_2)$ for all $(\Gamma,l)$, even if the leaf sets are not $S_1$ and $S_2$. Namely, then for at least one $b$ there exist $i,h\in S_b$ so that their first common descendant $v_{ih}$ is adjacent to the root. Hence $T_{ih}=T_0$, so the term $F_b^{\{i\}}$ is evaluated on a 0-valued face, and is thus the constant loop at $*$. In other words, \eqref{eq:lhs} factors as
    % \begin{equation}\label{eq:factor}
    %     \S^1\wedge\GPC_{\ul{n}} \ra 
    %     \S^1\wedge \Big(\faktor{\GPC_{\ul{n}}}{\GPC_{\ul{n}}^{\neq S_1, S_2}}\Big) \ra 
    %     \bigwedge_{i\in\ul{n}}Y_i
    % \end{equation}
    % where $\GPC_{\ul{n}}^{\neq S_1, S_2}\subseteq\GPC_{\ul{n}}$ consists of weighted trees that are \textit{not} of the shape $\Gamma=[\Gamma_1,\Gamma_2]$ for trees with leaf sets $S_1$ and $S_2$ in some order. This agrees with the output of $UR(F_1,F_2)$ by the definition of $[-,-]_{\graft}$.

\textbf{Right-Up.}
    Consider now the composite $RU$ of going to the right and then up. For $F_b=\{F_b^{T_b}\}_{T_b\subseteq S_b}\in \Omega\tofib(Y_{S_b};\kappa)$ as above, denote $G\coloneqq[F_1,F_2]_{\ext}$. Then $RU(F_1,F_2)$ is the map $\S^1\wedge\GPC_{\ul{n}}\to \bigwedge_iY_i$ given by:
\begin{equation}\label{eq:RU}
        \theta\wedge (\Gamma,l)\mapsto 
        \bigwedge_{i\in \ul{n}}G^{\{i\}}\big((T_{ih})_{h\in \ul{n}\sm\{i\}}\big)(\theta).
\end{equation}
    If $(\Gamma,l)$ has $T_0=0,1$ then $(\Gamma,l)=*$ and this is mapped to $*$. Consider now $(\Gamma,l)$ with $T_0\in(0,1)$. Then setting $\tau=T_0$ in Lemma~\ref{lem:F-i} gives a homotopy of the last map to
\[\theta\wedge (\Gamma,l) \mapsto
        \bigwedge_{i\in S_1} 
            [F_1^{\{i\}}
            \big(
                (\tfrac{T_{ih}-T_0}{1-T_0})_{h\in S_1\sm\{i\}}
            \big),\const_*]^{\downarrow \min_{h\in S_2}T_{ih}}(\theta)
        \wedge 
        \bigwedge_{i\in S_2}
            [\const_*,F_2^{\{i\}}
            \big(
                (\tfrac{T_{ih}-T_0}{1-T_0})_{h\in S_2\sm\{i\}}
            \big)]^{\downarrow \min_{h\in S_1}T_{ih}}(\theta).
\]
    If the leaf sets are not $S_1$ and $S_2$, this outputs $*$. Namely, then for at least one $b$ there exist $i,h\in S_b$ so that their first common descendant $v_{ih}$ is adjacent to the root. Hence $T_{ih}=T_0$, so the term $F_b^{\{i\}}$ is evaluated on a 0-valued face, and is thus the constant loop at $*$. %Thus, this map factors as \eqref{eq:factor} as well.

    On the other hand, if the leaf sets are $S_1$ and $S_2$, then $\min_{h\in S_b}T_{ih}=T_0$ and our formula becomes 
\[
    \theta\wedge (\Gamma,l) \mapsto
        \bigwedge_{i\in S_1} 
            [F_1^{\{i\}}
            \big(
                (\tfrac{T_{ih}-T_0}{1-T_0})_{h\in S_1\sm\{i\}}
            \big),\const_*]^{\downarrow T_0}(\theta)
        \wedge 
        \bigwedge_{i\in S_2}
            [\const_*,F_2^{\{i\}}
            \big(
                (\tfrac{T_{ih}-T_0}{1-T_0})_{h\in S_2\sm\{i\}}
            \big)]^{\downarrow T_0}(\theta).
\]
    By the computation from Example~\ref{ex:cube-deg-2-compute} this is homotopic to
\begin{equation}\label{eq:rhs}
    \theta\wedge (\Gamma,l) \mapsto
    \bigwedge_{i\in S_1}
                F_1^{\{i\}}\big((\tfrac{T_{ih}-T_0}{1-T_0})_{h\in S_1\sm\{i\}}\big)
                (T_0)
            \wedge 
            \displaystyle\bigwedge_{i\in S_2}
                F_2^{\{i\}}\big((\tfrac{T_{ih}-T_0}{1-T_0})_{h\in S_2\sm\{i\}}\big)
                (\theta).
\end{equation}
% Therefore, we see that 
% In addition, we claim that this formula holds in the limiting cases $T_0=0,1$ as well. Namely, by our check that $J$ is well defined, in these cases there exist leaves $i,h$ for which $T_{ih}=T_0=0,1$, so the formula~\eqref{eq:RU} outputs $*$. This is also clear for our formula if the leaf sets are not $S_1$ and $S_2$. Assume that 
% This is defined for $T_0=0$ and it also outputs $*$, because $\gamma^{\downarrow0}=\const_*$ by Definition~\ref{def:canonical}. On the other hand, when $T_0$ approaches $1$, then all $T_{ih}\geq T_0$ approach $1$.

\textbf{RU is homotopic to UR.}
Therefore, RU is the map~\eqref{eq:RU}, that is homotopic to~\eqref{eq:rhs} on the locus of points for which $T_0\in(0,1)$ and the leaf sets are $S_1$ and $S_2$; on the complement RU is equal to $*$. But the map~\eqref{eq:lhs} is exactly \eqref{eq:rhs} extended by $*$ on the complementary locus, showing that the homotopy of the previous sentence can be extended to all points and that RU is homotopic to \eqref{eq:lhs}. But we have also seen that \eqref{eq:lhs} is homotopic to UR.

In other words, we have shown that for a pair $(F_1,F_2)\in \Omega\tofib(Y_{S_1};\kappa)\tm\Omega\tofib(Y_{S_2};\kappa)$ the maps $\S^1\wedge \GPC_{\ul{n}}\to\bigwedge_{i\in\ul{n}} Y_i$ given as $UR(F_1,F_2)$ and $RU(F_1,F_2)$ are homotopic. It remains to observe that these homotopies are continuous in $(F_1,F_2)$, and that they extend over the ``nullhomotopies on the wedge'', giving the desired homotopy of maps $ \Omega \tofib(Y_{S_1};\kappa)\mywedge \Omega \tofib(Y_{S_2};\kappa)\to \Omega \Map_*(\GPC_{\ul{n}},\bigwedge_{i\in \ul{n}}Y_i)$.
\end{proof}

\begin{rem}
    The preceding proof should be compared to~\cite[Proposition~6.8]{Johnson}. 
\end{rem}

\begin{ex}
    Let us verify Theorem~\ref{thm:J} on the simplest example, $n=2$, $S_1=\{1\}$, $S_2=\{2\}$. We have $\GPC_{\{1\}}=\GPC_{\{2\}}=\S^0$ and $J_{\{1\}}=J_{\{2\}}=\Id$, and the square~\eqref{eq:Lie-vs-ext} becomes
    \[\begin{tikzcd}[column sep=large]
        \Omega Y_1\wedge \Omega Y_2
            \rar{[-,-]_{\graft}}
        & \Omega \Map_*(\GPC_{\ul{2}},Y_1\wedge Y_2)
        \\
        \Omega Y_1\mywedge \Omega Y_2
            \rar{[-,-]_{\ext}}
            \uar{\wt{q}}
        & \Omega \tofib(Y_1\vee Y_2;\kappa)
            \uar[swap]{\Omega J_{\ul{2}}}.
    \end{tikzcd}
    \]
    
    For $(F_1,F_2)\in\Omega Y_1\tm \Omega Y_2$ going up and to the right we get the map $[F_1,F_2]_{\graft}\colon\S^1\wedge\GPC_{\ul{2}}\cong\S^1\wedge\S^1\to \Omega(Y_1\wedge Y_2)$ given by $\theta\wedge T_0\mapsto F_1(T_0)\wedge F_2(\theta)$ thanks to Lemma~\ref{lem:square-htpy}. 
    
    On the other hand, thanks to Example~\ref{ex:cube-deg-2} the point $[F_1,F_2]_{\ext}\in\Omega \hofib(Y_1\vee Y_2\to Y_1\tm Y_2)$ is given by $[F_1,F_2]\in \Omega(Y_1\vee Y_2)$ together with paths $I\to\Omega Y_b$ given by $[F_1,*]^{\downarrow t_2}$ and $[*,F_2]^{\downarrow t_1}$ for respectively $b=1$ and $b=2$. The map $\Omega J_{\ul{2}}$ sends $[F_1,F_2]_{\ext}$ to the map $\theta\wedge T_0\mapsto [F_1,*]^{\downarrow T_0}(\theta)\wedge [*,F_2]^{\downarrow T_0}(\theta)$. By Example~\ref{ex:cube-deg-2-compute} this is homotopic to $\theta\wedge T_0\mapsto F_1(T_0)\wedge F_2(\theta)$, as desired.
\end{ex}

\subsection{Collapsing cubes of suspensions}
\label{subsec:J-collapsing}

We now specialise to the case $Y_i=\Sigma X_i$.
\begin{thm}\label{thm:J-collapsing}
    For any collection $\{X_i\}_{i\in\ul{n}}$ the following composite is homotopic to the canonical map:
    \[
        \prod_{\sigma\in \Sigma_{n-1}} w_\sigma(\{X_i\})
            \xrightarrow{\prod w_\sigma(\{x_i\})^\bull}
        \Omega\tofib\left( \Sigma X_{\ul{n}}; \kappa\right)
            \xrightarrow{\Omega J}
        \Omega\Map_*(\GPC_{\ul{n}}, \bigwedge_{i\in\ul{n}} \Sigma X_i)
            \xrightarrow[\sim]{\prod -\circ i_\sigma}
        \prod_{\sigma\in \Sigma_{n-1}}\Omega^n\Sigma^n w_\sigma(\{X_i\}),
    \]
    where $i_\sigma\colon\Delta^{n-1}\to\GPC_{\ul{n}}$ sends $(0\leq \theta_1\leq\dots\leq \theta_{n-1}\leq 1)$
    to the weighted tree in which the internal vertex adjacent to the leaf $\sigma(i)$ has distance $\theta_{\sigma(i)}$ to the root.
\end{thm}
\begin{proof} 
    Fix $\wedge_i z_i\in w(\{X_i\})$ and denote $G_w\coloneqq\Omega J \big(w(\{x_i\})^\bull(\wedge_iz_i)\big)\colon\GPC_{\ul{n}}\to\Omega\bigwedge_{i\in\ul{n}}\Sigma X_i$ for a Lie word $w$.
    By its definition in~\eqref{eq:total-Samelson} the total Samelson map $w(\{x_i\})^\bull=[w_1(x_i)^\bull,w_2(x_i)^\bull]_{\ext}$ is an external bracket, so Theorem~\ref{thm:J} implies 
    $G_w=\big[G_{w_1},\, G_{w_2}\big]_{\graft}$. 
    By the definition of the grafting bracket and Lemma~\ref{lem:square-htpy}, for a weighted tree $(\Gamma,l)$ with $\Gamma=[\Gamma_1,\Gamma_2]$ and the root edge $e_0$ we have
    \[
    G_w(\Gamma,l)(\theta)
        =G_{w_1}(\Gamma_1,l_1)(l(e_0))\wedge G_{w_2}(\Gamma_2,l_2)(\theta).
        \,
    \]
    We need to prove that for $(\Gamma,l)=i_\sigma(\theta_1,\dots, \theta_{n-1})$ this is homotopic to the canonical map
    \[
        (\theta_1,\dots,\theta_{n-1},\theta)\mapsto 
        \Big(\bigwedge_{i=1}^{n-1} (\theta_i\wedge z_i)\Big)
        \wedge (\theta\wedge z_n).
    \]
    We do this by induction on $n\geq1$. 
    For the base case $n=1$ we have $\GPC_{\ul{1}}=\S^0$ and the map from the statement is 
    \[
        X_1\ra\Omega\hofib(\Sigma X_1\to *)\ra\Omega\Map_*(\S^0,\Sigma X_1)\simeq\Omega\Sigma X_1,
    \]
    given by $z_1\mapsto (x_{\sigma1}(z_{\sigma1}); \const_*)\mapsto x_{\sigma1}(z_{\sigma1})=(\theta\mapsto \theta\wedge z_1)$, as desired. 
    
    Now fix $w_\sigma\in B\Lie(n)$ for $\sigma\in\Sigma_{n-1}$ as in \eqref{eq:right-normed} and  $(\Gamma,l)=i_\sigma(\theta_1,\dots, \theta_{n-1})$. 
    Thus, $\Gamma_1$ has one leaf, $l(e_0)=\theta_{\sigma1}$ and $(\Gamma_2,l_2)=i_{\sigma|_{S_2}}(\tfrac{\theta_i-\theta_{\sigma1}}{1-\theta_{\sigma1}})$. Then the induction hypothesis implies
    \begin{align*}
        G_{w_\sigma}(i_\sigma(\theta_1,\dots,\theta_{n-1}))(\theta)
        &=G_{w_1}(\Gamma_1,l_1)(\theta_{\sigma1})\wedge G_{w_2}(\Gamma_2,l_2)(\theta)\\
        &=\theta_{\sigma1}\wedge z_{\sigma1}\wedge\bigwedge_{i\in\ul{n}\sm\sigma1} (\tfrac{\theta_i-\theta_{\sigma1}}{1-\theta_{\sigma1}}\wedge z_i)\wedge (\theta\wedge z_n).
    \end{align*}
    Finally,  the homotopy as in \eqref{eq:lhs} shows this is homotopic to the desired canonical map.
\end{proof}

We complete the goal explained at the beginning of §\ref{sec:improve} by combining this theorem with Corollary~\ref{cor:hm}.
\begin{cor}\label{cor:range}
    Given $c$-connected spaces $X_i$ for $i=1,\dots,n$ and some $c\geq0$, the map 
    \[
        J\colon\tofib(\Sigma X_{\ul{n}}; \kappa)
    \ra \Map_*(\GPC_{\ul{n}}, \bigwedge_{i\in\ul{n}} \Sigma X_i)
    \]
    is $(n+1)(c+1)$-connected and $\Sigma_n$-equivariant. 
\end{cor}

Recall from Lemma~\ref{lem:Lie-d(n)} that $\Lie_D(n)\cong\Lie(n)\otm(\sgn_n)^{\otm D}$. Then Example~\ref{ex:Lie} is improved to the following; see \cite[Corollary~13.9]{Rognes} for a different proof.
\begin{cor}\label{cor:ex-Lie}
    For the total homotopy fibre of the collapsing cube of $\S^{d-1}_{\ul{n}}=\bigvee_{n}\S^{d-1}$ the first nonvanishing homotopy group admits a $\Sigma_n$-equivariant isomorphism
    \[
        \pi_{n(d-2)}\Omega\tofib\left(\S^{d-1}_{\ul{n}}; \kappa\right)
        \cong\Lie_{d-2}(n).\qedhere
    \]
\end{cor}
\begin{proof}
   In Example~\ref{ex:Lie} we saw that $\pi_{n(d-2)}\Omega\tofib(\S^{d-1}_{\ul{n}}; \kappa) \cong \pi_{n(d-2)}\Omega\Map_*(\GPC_{\ul{n}}, \bigwedge^n\S^{d-1})$ is the first nonvanishing homotopy group. Therefore, by the Hurewicz Theorem this equals the homology group $H_{n(d-2)+1}\Map_*(\GPC_{\ul{n}}, \S^{n(d-1)})$ (with $\Z$ coefficients). The Freudenthal Suspension Theorem implies:
\begin{align*}
\pi_{n(d-2)}\Omega\tofib(\S^{d-1}_{\ul{n}}; \kappa) 
        \cong H_{n(d-2)+1}\Sigma^\infty\Map_*(\GPC_{\ul{n}}, \S^{n(d-1)})
        &\cong \wt{H}_{n(d-2)+1}((\Sigma^\infty\GPC_{\ul{n}})^\vee\wedge \S^{n(d-1)})\\
        &\cong \wt{H}_{1-n}((\Sigma^\infty\GPC_{\ul{n}})^\vee)\otm\wt{H}_{n(d-1)}(\S^{n(d-1)}),
\end{align*}
    where $(\Sigma^\infty\GPC_{\ul{n}})^\vee$ is the Spanier--Whitehead dual of the suspension spectrum of $\GPC_{\ul{n}}$.
    Using \eqref{eq:coh-Tn} we have $\wt{H}_{1-n}((\Sigma^\infty\GPC_{\ul{n}})^\vee)\cong \wt{H}^{n-1}(\Sigma^\infty\GPC_{\ul{n}})\cong\Lie(n)\otm\sgn_n$, whereas $\wt{H}_{n(d-1)}(\S^{n(d-1)})\cong H_n(\S^n)^{\otm (d-1)}\cong \sgn_n^{\otm (d-1)}$. Therefore, $\pi_{n(d-2)}\Omega\tofib(\S^{d-1}_{\ul{n}}; \kappa)$ is isomorphic to $\Lie(n)\otm\sgn_n^{\otm d}\cong\Lie_{d-2}(n)$.
\end{proof}
\begin{rem}
   There is a mistake in the cited work~\cite{Rognes}, because it is overlooked that for $m+n$ even, $\pi_*\Omega(\bigvee_k\S^{m+n})$ is a graded Lie algebra in which $x_i\in\pi_{m+n}\S^{m+n}\subset\pi_{m+n}(\bigvee_k\S^{m+n})\cong\pi_{m+n-1}\Omega(\bigvee_k\S^{m+n})$ has odd degree. Thus, the first homotopy group of $\Omega\tofib(\bigvee_k\S^{m+n};\kappa)$ is $\Lie(n)\otm\sgn_n$, instead of $\Lie(n)$ as stated in \cite[Corollary~13.9]{Rognes}. All other results \cite[Lemma~13.11, Corollary~13.13]{Rognes} hold without change.
\end{rem}
\begin{rem}\label{rem:grafting-Lie}
    Note that on the first nonvanishing homotopy groups the external and grafting brackets correspond to grafting of Lie trees: $\Lie_{d-2}(n_1)\tm \Lie_{d-2}(n_2)\to \Lie_{d-2}(n_1+n_2)$.
\end{rem}

%%%%%%%%%%%
\subsection{Collapsing cubes more generally}
\label{subsec:J-general}

Assume now that we have a general space $Y_0=M$ but that $Y_i=\Sigma X_i$ are still suspensions for $i\in\ul{n}$. Then there is a generalisation of the Hilton--Milnor equivalence~\eqref{eq:HM}, discussed in~\cite[Section~5.1]{K-thesis-paper}, that we now recall. 

First, there is a classical fibration sequence from, for example,~\cite{Gray}:
\[
     \Omega\Sigma(X_{\ul{n}}\wedge(\Omega M)_+)\ra \Omega(M\vee \Sigma X_{\ul{n}})
    \ra \Omega M.
\]
Note that the fibre is equivalent to $\Omega(\bigvee_{i\in \ul{n}} \Sigma (X_i\wedge(\Omega M)_+))$, and to this the Hilton Milnor theorem~\eqref{eq:HM-all} applies. Therefore, let us put $X_i^M\coloneqq X_i\wedge(\Omega M)_+$
and consider the fibration sequence
\begin{equation}\label{eq:tofib-fib-seq}
    \Omega\tofib\left(\Sigma X^M_{\ul{n}}; \kappa\right)\ra 
        \Omega\tofib\left(M\vee\Sigma X_{\ul{n}}; \kappa\right)\ra
        \Omega\tofib\left(\const_{M}\right).
\end{equation}
The base is the total homotopy fibre of the constant cube $S\mapsto M$, which is contractible. Hence, the fibre and total space are homotopy equivalent:
\begin{equation}\label{eq:tofib-non-susp}
    \Omega\tofib\left(\Sigma X^M_{\ul{n}}; \kappa\right)\xrightarrow{\sim}
        \Omega\tofib\left(M\vee\Sigma X_{\ul{n}}; \kappa\right),
\end{equation}
and~\eqref{eq:HM} implies
\begin{equation}\label{eq:HM-gen}
        hm_{NB}\colon\;
            \prod_{w\in NB\LL(n)}\Omega\Sigma w(\{X_i^M\})\xrightarrow{\sim}
            \Omega\tofib\left(\Sigma X^M_{\ul{n}}; \kappa\right)\simeq\Omega\tofib\left(M\vee\Sigma X_{\ul{n}}; \kappa\right).
\end{equation}
In particular, assuming $X_i$ are $c$-connected for some $c\geq0$, the spaces $X_i^M$ are $c$-connected as well. We can again restrict to words of length exactly $n$, for which $w(\{X_i^M\})\simeq\bigwedge_{i\in\ul{n}} X_i^M$.
The following generalises Theorem~\ref{thm:J-collapsing} and Corollary~\ref{cor:range}.
\begin{thm}\label{thm:J-general}
For spaces $M$ and $\{X_i\}_{i\in\ul{n}}$ denote $X_i^M\coloneqq X_i\wedge(\Omega M)_+$.
There is a $\Sigma_n$-equivariant map
    \[
        J^M\colon\tofib\left(M\vee \Sigma X_{\ul{n}}; \kappa\right)
        \ra \Map_*
        ( \GPC_{\ul{n}},\bigwedge_{i\in\ul{n}}\Sigma X_i^M)
    \]
so that $(\Omega J^M\circ w(\{x_i^M\}))\circ i_\sigma\colon w(\{X_i^M\})\to\Omega^n\Sigma^n w(\{X_i^M\})$ is the canonical map for any $w_\sigma\in B\Lie(n)$.

Moreover, if for some $c\geq0$ each $X_i$ is $c$-connected, then $J^M$ is $(n+1)(c+1)$-connected.
    The first nonvanishing homotopy group admits a $\Sigma_n$-isomorphism
    \[
        \pi_{n(c+1)} \Omega\tofib\left(M\vee\Sigma X_{\ul{n}}; \kappa\right)
        \cong   \Lie(n)\otm \sgn_n\otm
        \bigotimes_{i=1}^n\big(H_{c+1}X_i\otm\Z[\pi_1M]\big).
    \]
\end{thm}
\begin{proof}
    The inclusion~\eqref{eq:tofib-non-susp} is a $\Sigma_n$-equivariant homotopy equivalence. Its homotopy inverse can be made explicit~\cite{K-thesis-paper}, and composing it with $J\colon\tofib(\Sigma X^M_{\ul{n}}; \kappa)\to\Map_*( \GPC_{\ul{n}}, \bigwedge_{i\in\ul{n}}\Sigma X_i^M)$ from Theorem~\ref{thm:J} gives the map $J^M$. The rest of the first sentence follows from Theorem~\ref{thm:J-collapsing}.

    For the last two sentences, we argue similarly as in Corollary~\ref{cor:ex-Lie}, namely:
\begin{align*}
\pi_{n(c+1)}\Omega\tofib\left(M\vee\Sigma X_{\ul{n}}; \kappa\right)
        &\cong H_{n(c+1)+1}\Sigma^\infty\Map_*(\GPC_{\ul{n}}, \wedge_{i=1}^n\Sigma X_i^M)\\
        &\cong \wt{H}_{1-n}(\Sigma^\infty\GPC_{\ul{n}}^\vee)\otm\wt{H}_{n(c+2)}(\wedge_{i=1}^n\Sigma X_i^M)\\
        &\cong \wt{H}^{n-1}(\Sigma^\infty\GPC_{\ul{n}})\otm\otimes_{i=1}^n\wt{H}_{c+2}(\Sigma X_i^M)\\
        &\cong \Lie(n)\otm \sgn_n
        \otimes_{i=1}^n(H_{c+1}X_i\otm\Z[\pi_1M]).\qedhere
\end{align*}
\end{proof}
\begin{ex}\label{ex:Lie-pi}
    Similarly to Corollary~\ref{cor:ex-Lie}, if $X_i=\S^{d-2}$ for each $i\in\ul{n}$, then $X_i^M=\Sigma^{d-2}(\Omega M)_+$ and $c=d-3$, and there is a $\Sigma_n$-isomorphism of abelian groups
\[
        \pi_{n(d-2)}\Omega\tofib\big(M\vee\S^{d-1}_{\ul{n}}; \kappa\big)
        \cong \Lie_{d-2} 
        (n)\otm\Z[(\pi_1M)^n].
\]
    We think of an element here as a linear combination of trees $(\Gamma,g_{\ul{n}})$ that have leaf $i$ decorated by an element $g_i\in\pi_1M$, and $\Sigma_n$ acts by permuting decorated leaves. 
    Such $(\Gamma,g_{\ul{n}})$ corresponds to the homotopy class of a map $\Gamma\to M$, together with a path $\gamma_i$ from the basepoint of $M$ to the image of the $i$-th univalent vertex, including the root $i=0$. Then $g_i=[\gamma_0^{-1}\gamma_{0i}\gamma_i]$, for the unique path $\gamma_{0i}$ in the tree from $0$ to~$i$.
\end{ex}

%%%%%%%%%%%%%%%%%%%%%%%%%%%%%%%%%%%%%%%%%%%%%%%%%%%%%%%%%%%%%%%%%%
\section{Calculi}\label{sec:calculi}
We discuss homotopy calculus in §\ref{subsec:htpy-calc}, and isotopy calculus in §\ref{subsec:istpy-calc}.

%%%%%%%%%%%
\subsection{Homotopy calculus (of the identity functor)}
\label{subsec:htpy-calc}

\subsubsection{The layers}

Goodwillie's~\cite{GooIII} classification of homogeneous functors of degree $n$ says that the layer $D_nI\coloneqq\hofib(P_nI\to P_{n-1}I)$ of the identity functor applied to a space $X$ is given by
\begin{equation}\label{eq:layers-1}
    D_nI(X)\cong \Omega^\infty (D^{(n)}I(X,\dots,X)_{h\Sigma_n}).
\end{equation}
This is an infinite loop space of a homogeneous functor with values in spectra, obtained by the process called multilinearisation from the cross effect functor $cr_nI$. Namely, $cr_nI$ is the symmetric multilinear functor of $n$ variables $X_1,\dots,X_n$ defined as
\[
    cr_nI(X_1,\dots,X_n)\coloneqq\tofib\big(\bigvee_{i\in\ul{n}}X_i;\kappa\big),
\]
the total homotopy fibre of the collapsing cube from~\eqref{eq:collapsing-cube}. We define its multilinearisation as
\[
    D^{(n)}I(X_1,\dots,X_n)\coloneqq B^\infty\operatorname*{hocolim}\limits_{k\geq0}
        \Omega^{nk}\tofib(\bigvee_{i\in\ul{n}}\Sigma^k X_i;\kappa).
\]
% In fact, one can multilinearize $cr_n(\Omega\Sigma)$ instead.
% Explicitly, we have
% \[
%     cr_n\Omega\Sigma(X_1,\dots,X_n)\coloneqq
%     \tofib\left(\bigvee_{i\in \ul{n}}\Omega\Sigma X_i; \kappa\right)
%     =\Omega\tofib\left(\bigvee_{i\in \ul{n}}\Sigma X_i; \kappa\right),
% \]
Using the Hilton--Milnor Corollary~\ref{cor:hm} for $0\leq m\leq (n+1)(k+1)-k$ we have isomorphisms
\begin{align*}
    \pi_m \Omega^{nk}\tofib(\bigvee_{i\in \ul{n}}\Sigma^k X_i; \kappa)
    & \cong   \pi_m \prod_{(n-1)!} \Omega^{nk}\Sigma \bigwedge_{i\in\ul{n}}\Sigma^{k-1} X_i\\
    & \cong   \pi_m \prod_{(n-1)!} \Omega^{nk}\Sigma^{nk-n+1} \bigwedge_{i\in\ul{n}} X_i
    \cong   \pi_m \prod_{(n-1)!} Q\Sigma^{1-n} \bigwedge_{i\in\ul{n}} X_i\,,
\end{align*}
implying that
\[
    D^{(n)}I(X_1,\dots,X_n)\simeq 
    \prod_{(n-1)!}\Sigma^\infty\Sigma^{1-n}\bigwedge_{i\in\ul{n}} X_i\simeq \Map_*\big(\bigvee_{(n-1)!}\S^{n-1},\Sigma^\infty\bigwedge_{i\in\ul{n}} X_i\big).
\]
However, as mentioned in §\ref{subsec:ext-Sam-brackets}, the Hilton--Milnor equivalence forgets the $\Sigma_n$-action. We have seen how this can be solved in Theorem~\ref{thm:J-collapsing}: we replace $\bigvee_{(n-1)!}\S^{n-1}$ by the homotopy equivalent partition complex $\GPC_{\ul{n}}$ and obtain a $\Sigma_n$-equivalence
\begin{equation}\label{eq:multi}
    D^{(n)}I(X_1,\dots,X_n)
        \simeq \Map_*(\GPC_{\ul{n}},\Sigma^\infty\bigwedge_{i\in\ul{n}} X_i).
\end{equation}
Therefore, \eqref{eq:layers-1} says that the $n$-th layer of the Taylor tower of the identity at $X$ is the infinite loop space on the homotopy orbits of the value of the last functor on $X_1=\dots=X_n\eqqcolon X$, that is:
\begin{equation}\label{eq:layer-2}
    D_nI(X)=\Omega^\infty\Map_*(\GPC_{\ul{n}},\Sigma^\infty X^{\wedge n})_{h\Sigma_n}.
\end{equation}
The symmetric group acts on $\GPC_{\ul{n}}$ as in Definition~\ref{def:GPC} and on $X^{\wedge n}$ by permuting the factors.

\subsubsection{A bracket on the layers}

In particular, putting $X_i=\S^0$ in \eqref{eq:multi} we obtain 
\[
    \partial_nI\coloneqq D^{(n)}I(\S^0,\dots,\S^0)\cong(\Sigma^\infty\GPC_{\ul{n}})^\vee,
\]
called the \emph{derivative of the identity functor}. Here we recognise precisely the shifted spectral Lie operad $\mathbf{s}\oLie$ that was defined in~\eqref{eq-def:s-spectral-Lie}. Then
\eqref{eq:layer-2} for any $X$ becomes
\begin{equation}\label{eq:layer-3}
    D_n(X)=\Omega^\infty((\Sigma^\infty\GPC_{\ul{n}})^\vee\wedge\Sigma^\infty X^{\wedge n})_{h\Sigma_n}
    =\Omega^\infty(\mathbf{s}\oLie(n) \wedge(\Sigma^\infty X)^{\wedge n})_{h\Sigma_n}.
\end{equation}
Thus, the spectra associated to the layers assemble into the free shifted spectral Lie algebra on the suspension spectrum of $X$:
\[
    \bigvee_{n\geq1}B^\infty D_n(X)\coloneqq \bigvee_{n\geq1}(\mathbf{s}\oLie(n) \wedge(\Sigma^\infty X)^{\wedge n})_{h\Sigma_n}=\mathrm{Free}_{\mathbf{s}\oLie}(\Sigma^\infty X).
\]
\begin{rem}
     This description of the layers of the identity is due to \cite{Arone-Mahowald}. The connection of partition complexes to bar constructions is from~\cite{Ching}.
\end{rem}

% \begin{rem}
% Moreover, \cite{AK-Snaith} use a version of \emph{Snaith splitting} $\Sigma^\infty Q^kX\simeq\bigvee_{i=1}^\infty\Sigma^\infty(X^{\wedge i})_{h\Sigma_i}$ to study $D_nQ^kX$.
% For each n the functor Qn+1 has a split Taylor tower, which can be read off from the Snaith splitting formula.
% \end{rem}

In fact, the terms of the operad $\oLie$ itself also appear as Goodwillie differentials, namely of the functor $F=\Omega\Sigma$, and one can in analogy to the above arguments show that
\[
    \oLie(n)=\partial_n(\Omega\Sigma)=\Map_{Sp}(\S^1\wedge(\Sigma^\infty\GPC_{\ul{n}}),(\S^1)^{\wedge n}),
\]
with $H_0\oLie(n)\cong \wt{H}^{n-1}(\GPC_{\ul{n}})\otm\sgn_n\cong \Lie(n)$. Moreover, for the layers we have
\begin{align*}
    D_n(\Omega\Sigma)(X)  
        &=\Omega^\infty \Omega\Map_*(\GPC_{\ul{n}},\Sigma^\infty(\Sigma X)^{\wedge n})_{h\Sigma_n},
       \\
    % \pi_{n(c+1)}D_n(\Omega\Sigma)(X) 
    %     &= H_{n(c+1)}D_n(\Omega\Sigma)(X) = \pi_{n(c+1)}Q( \Omega\Map_*(\GPC_{\ul{n}},(\Sigma X)^{\wedge n})_{h\Sigma_n})\\
    \bigvee_{n\geq1}B^\infty D_n(\Omega\Sigma)(X) 
        &\coloneqq \bigvee_{n\geq1}(\oLie(n)\wedge(\Sigma^\infty X)^{\wedge n})_{h\Sigma_n}
        =\mathrm{Free}_{\oLie}(\Sigma^\infty X).
\end{align*}
% Note that $\partial_nI= (\S^{-1})^{\wedge n}\wedge \Sigma\partial_n(\Omega\Sigma)$.
The advantage of this non-shifted version is that we have a bracket
\[
    [-,-]_{\oLie}\colon D_{n_1}(\Omega\Sigma)(X) \wedge D_{n_2}(\Omega\Sigma)(X)\to D_{n_1+n_2}(\Omega\Sigma)(X).
\]
Namely, an algebra $A$ over an operad $O$ is in particular equipped with binary operations $O(2)\otm A\otm A\to A$. We apply this to the algebra $A=\bigvee_{n\geq1}D_n(\Omega\Sigma)(X)$ over $O=\oLie$. Note that $\oLie(2)=\S^0$, so we have the claimed bracket, which we can call the \emph{operadic Lie bracket}.

So far we have seen many other brackets -- external Samelson, grafting, and operadic, and now we can relate all of them.
\begin{proof}[Proof of Theorem~\ref{mainthm:C}]
    For the topmost square, both the operad structure on $\oLie$ and the grafting bracket on the spaces $\Omega \Map_*(\GPC_{S_1},\bigwedge_{i\in S_1}\Sigma X_i)$ come from the cooperad structure on $\GPC_{\ul{n}}$ from \eqref{eq:cooperad}, 
    and it is easy to see that they agree under the infinite suspension map.
    For the middle square, Theorem~\ref{thm:J} shows that the grafting bracket is compatible with the external Samelson bracket on total homotopy fibres. The bottom square commutes by the definition in~\eqref{eq:total-Samelson}. Finally, the triangle commutes by Lemma~\ref{lem:total-vs-canonical}, that gives the compatibility of total and canonical Samelson maps.
\end{proof}

%%%%%%%%%%%%%%%%%%%%%%%%%%%%%%%%%%%%%%%%%%%%%%%%%%%%%%%%%%%%%%%%%%
\subsection{Isotopy calculus (for embeddings of the arc)}
\label{subsec:istpy-calc}

\subsubsection{The setup}

We fix $d\geq3$ and a smooth $d$-dimensional manifold $M$ with nonempty boundary. Fix also a neat arc $\unknot\colon\D^1\hra M$, that is, a smooth embedding meeting $\partial M$ transversely in $\unknot(\D^1)\cap\partial M=\unknot(\partial\D^1)$. Let $\nu\unknot$ be a fixed tubular neighbourhood of $\unknot$.

Finally, fix a collection of disjoint subintervals $ J_i=[L_i,R_i]\subseteq \D^1$ with $0<L_i<R_i<L_{i+1}$, such that the sequence $R_i$ converges to $R_\infty<1$.

\begin{defn}\label{def:M-T}
    For any $n\geq1$ and a set $T\subseteq[n]\coloneqq\{0,1,\dots,n\}$ we define the \emph{punctured manifold}
    \[
        M_T\coloneqq (M\sm\nu\unknot)\cup\nu\unknot\big(\bigsqcup_{i\in T} J_i\big)
        =M\sm\nu\unknot\big(\D^1\sm \bigsqcup_{i\in T} J_i\big),
    \]
    and for any $T\subseteq T'$ the obvious inclusion map $i_{T,T'}\colon M_T\hra M_{T'}$. 
\end{defn}
In words, $M_T$ is $M$ minus the thickened punctured arc $\nu\unknot(\D^1\sm \bigsqcup_{i\in T} J_i)$; see Figure~\ref{fig:M-emptyset} for two examples. Equivalently, we puncture $M$ at the pieces of $\nu \unknot$ that are complementary to $T$. 
\begin{figure}[!htbp]
    \centering
    \includegraphics[width=0.81\linewidth]{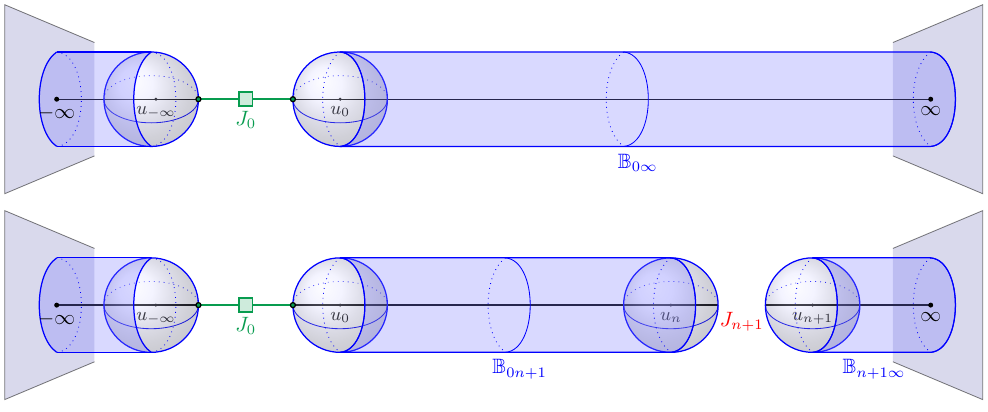}
    \caption{The manifold $M_T$, for $T=\{0\}$ and $T=\{0,n+1\}$ respectively, is the complement of the blue material. An element of $\Embp(J_0,M_{0S})$ is an arc embedded in $M_{0S}$ with the same endpoints as $J_0$. The map $\rho_S^k$ for $S=\emptyset$ and $k=n+1$ is simply the inclusion of the arcs in the top to those in the bottom picture. Reproduced from~\cite{K-thesis-paper}.}
    \label{fig:M-emptyset}
\end{figure}
\begin{defn}
    For any $S\subseteq\ul{n}\coloneqq\{1,\dots,n\}$ define the space $\Embp(J_0,M_{0S})$
    as the space of neat embeddings $\D^1\hra M_{0S}$ which near the boundary agree with $\unknot|_{J_0}$.
    
    For $k\notin S\subseteq\ul{n}$ we write $Sk\coloneqq S\sqcup\{k\}$, and define the map
    \[
    \rho_S^k\colon\Embp(J_0,M_{0S})\to\Embp(J_0,M_{0Sk})
    \]
    as the postcomposition with the inclusion $i_{0S,0Sk}\colon M_{0S}\hra M_{0Sk}$.
\end{defn}
\begin{lem}\label{lem:knots-J-0}\label{lem:lbt}
    The inclusion and the unit derivative give the respective homotopy equivalences
    \begin{align*}
        \Embp(J_0,M_0)
        &\xhookrightarrow{\sim}\Embp(\D^1,M),\\
        \deriv\colon\Embp(J_0,M_{0Sn+1})
        &\xrightarrow{\sim} \Omega\S^{d-1}\tm\Omega\big(M\vee\S^{d-1}_{\ul{n}}\big).
    \end{align*}
\end{lem}
Recall that we use notation $\S^{d-1}_{\ul{n}}\coloneqq\bigvee_{i\in \ul{n}}\S^{d-1}$.
The second equivalence is known to experts under the name ``light bulb trick'', and is not hard to prove. It probably originates in the work of Cerf and Germain, and was used by Hatcher, Goodwillie and many others since then. In low dimensional topology, usually only the $\pi_0$-statement is considered, and is related to the Norman trick. This history is summarised in Section 2.3 of \cite{KT-highd}, where we also discuss some new applications of light bulb tricks.

\subsubsection{The layers}

The fibres $F_{n+1}(M)\coloneqq\hofib(T_{n+1}(M)\to T_n(M))$ in Goodwillie's punctured knots model of the Taylor tower for $\Embp(\D^1,M)$ can be described as~\cite{K-thesis-paper}:
    \[
        F_{n+1}(M)\coloneqq
        \tofib\big(\Embp(J_0,M_{0\ul{n}n+1});\rho\big).
    \]
Using that the maps $\rho^k_{Sn+1}$ are inclusions gives the following characterisation, as in \eqref{eq:tofib-pt}: a point $f\in F_{n+1}(M)$ is an $n$-parameter family of re-embeddings of $J_0$ in the maximally punctured manifold $M_{0\ul{n+1}}$, such that: $f$ is trivial on 0-valued faces, and on the 1-valued face labelled by $S$ it must miss $\unknot(\bigsqcup_{i\in S} J_i)$.

Note that each space in this cube is a loop space by the second equivalence in Lemma~\ref{lem:lbt}. However, $\rho^k_{Sn+1}$ \emph{do not} correspond to loop space maps. If they did, one could use Lemma~\ref{lem:deloop} to deloop the total homotopy fibre $n$ times. Nevertheless, one can first deloop the cube $n$ times directly, by proving an analogue of Lemma~\ref{lem:deloop} for embeddings, and then apply Lemma~\ref{lem:lbt}. This gives the following.
\begin{thm}[{\cite{K-thesis-paper}}]
\label{thm:K-thesis-paper}
    There are explicit homotopy equivalences 
    \begin{align*}
        F_{n+1}M =\tofib
            \big(\Embp(J_0,M_{0\ul{n}n+1}); \rho\big)
        \xrightarrow[\sim]{\chi}
        \Omega^n\tofib
            \big(\Embp(J_0,M_{0\ul{n}n+1}); \lambda\big)
        \xrightarrow[\sim]{\deriv} 
        \Omega^n\tofib
            \big(\Omega (M\vee\S^{d-1}_{\ul{n}}); \Omega \kappa\big).
    \end{align*}
\end{thm}

The final cube in Theorem~\ref{thm:K-thesis-paper} is the collapsing cube from §\ref{subsec:J-general} for $Y_0=M$, $Y_i=\Sigma X_i$, $X_i=\S^{d-2}$, so
% Therefore, implies
\[
        F_{n+1}(M) 
        \simeq \Omega^{n+1}\tofib
            \big(M\vee\S^{d-1}_{\ul{n}}; \kappa\big)
        \simeq\Omega^{n+1}\tofib
            \big(\bigvee_{\ul{n}} \S^{d-1}\wedge (\Omega M)_+; \kappa\big),
\]
by \eqref{eq:tofib-non-susp}. Moreover, we can use the Hilton--Milnor equivalence~\eqref{eq:HM-gen} to describe the homotopy type of $F_{n+1}(M)$ for any $M$ and $n$, as was done in~\cite[Theorem~B]{K-thesis-paper}.
\begin{rem}
    It is worth pointing out that the equivalence of Theorem~\ref{thm:K-thesis-paper} can quickly be obtained from the original definition of the Taylor tower by Goodwillie and Weiss~\cite[Section~5]{GW}. The key point of Theorem~\ref{thm:K-thesis-paper} is that the map $\chi$ is explicit, and can be used to compare geometric constructions in $\Embp(\D^1,M)$ via graspers, and algebraic constructions in $F_{n+1}$ via Samelson maps.
\end{rem}

In fact, using the generalised Johnson map $J$ from Theorem~\ref{thm:J-general}, we now have the following statement that includes the $\Sigma_n$-action; see also Example~\ref{ex:Lie-pi}.
\begin{thm}
\label{thm:Fn-J}
    For any $d$-manifold $M$ with boundary, $d\geq3$ and $n\geq1$, the layer $F_{n+1}(M)$ in the Taylor tower for $\Embp(\D^1,M)$ admits $(n+1)(d-2)$-connected maps 
    \[
        \Omega^n\prod_{w\in B\Lie(n)} w(\Sigma^{d-2}(\Omega M)_+)
        \xrightarrow{\prod w^{F_{n+1}}}
        F_{n+1}(M)
        \xrightarrow{J\circ\deriv\circ\chi}
        \Omega^{n+1}\Map_*(\GPC_{\ul{n}}, (\Sigma^{d-1}(\Omega M)_+)^{\wedge n}).
    \]
    Moreover, $J\circ\deriv\circ\chi$ is $\Sigma_n$-equivariant, where the action on $F_{n+1}(M) \coloneqq\tofib(\Embp(J_0,M_{0\ul{n}n+1});\rho)$ comes from the $\Sigma_n$-action on the cube $(\Embp(J_0,M_{0\ul{n}n+1});\rho)$ by permuting the labels in $\ul{n}$, whereas the action on $ \Omega^{n+1}\Map_*(\GPC_{\ul{n}}, (\Sigma^{d-1}(\Omega M)_+)^{\wedge n})=\Map_*((\S^1)^{\wedge \ul{n}}, \Omega \Map_*(\GPC_{\ul{n}}, (\Sigma^{d-1}(\Omega M)_+)^{\wedge\ul{n}}))$ is the diagonal action, using permutations on $(\S^1)^{\wedge \ul{n}}$ and $(\Sigma^{d-1}(\Omega M)_+)^{\wedge\ul{n}}$, and the $\Sigma_n$-action on the weighted trees in $\GPC_{\ul{n}}$.

    In particular, on $\pi_{n(d-3)}$ we have a pair of $\Sigma_n$-isomorphisms
    \[
        \pi_{n(d-3)}F_{n+1}(M)\cong\Lie_{d-1}(n)\otm\Z[(\pi_1M)^n].
    \]
\end{thm}
\begin{proof}
    We need only explain the $\Sigma_n$-equivariance.  Recall that a point of $F_{n+1}(M)$ is given by a map $F^{\ul{n}}\colon I^{\ul{n}}\to \Embp(J_0,M_{0\ul{n}n+1})$ satisfying certain conditions on faces. The action of $\Sigma_n$ is by permuting the coordinates of the cube and re-indexing on $M_{0\ul{n}n+1}$.
    
    We claim that $\chi$ is $\Sigma_n$-equivariant. In order to stay concise, we do not include here the full definition of $\chi$ from \cite[Appendix~A]{K-thesis-paper}, but only point out that the map $\S^n\to \Embp(J_0,M_{0\ul{n}n+1})$ is obtained by gluing $2^n$ copies of $I^{\ul{n}}\to \Embp(J_0,M_{0\ul{n}n+1})$, but modified using certain canonical reflections and homotopies. This is more concisely explained in \cite[Section~8.2]{K-families}. This means that permuting coordinates on $I^{\ul{n}}$ corresponds to permuting the smash factors of $\S^n\cong (\S^1)^{\wedge\ul{n}}$. Moreover, it is easy to check that $\chi^{-1}$, defined in Lemma~\ref{lem:deloop}, is $\Sigma_n$-equivariant.
    
    The derivative map $\deriv$ clearly translates re-indexing on $M_{0\ul{n}n+1}$ to re-indexing on $M\vee\S^{d-1}_{\ul{n}}$. Finally, $J$ is $\Sigma_n$-equivariant by Theorem~\ref{thm:J-general}.

    Using this and Example~\ref{ex:Lie-pi} we indeed have $\Sigma_n$-isomorphisms
    \begin{align*}
        \pi_{n(d-3)}F_{n+1}(M)
        &\cong 
            \wt{H}_n(\S^n)\otm 
            \wt{H}_{n(d-2)}(\Omega\Map_*(\GPC_{\ul{n}}, (\Sigma^{d-1}(\Omega M)_+)^{\wedge n}))\\
        &\cong 
            \sgn_n\otm \Lie_{d-2} (n)\otm\Z[(\pi_1M)^n]\\
        &\cong \Lie_{d-1}(n)\otm\Z[(\pi_1M)^n].\qedhere
    \end{align*}
\end{proof}

\subsubsection{Graspers}
\label{subsec:graspers}

In this section we describe classes in $\pi_k\Embp(\D^1,M)$ that correspond to the classes in $\pi_{n(d-3)}F_{n+1}(M)\cong\Lie(n)\otm\Z[(\pi_1M)^n]$, that is, realise the first nontrivial homotopy groups of the layers.

First observe that a nonzero class in $\pi_k\Embp(\D^1,M)$, where an arbitrary basepoint $\u\colon \D^1\hra M$ is omitted from the notation, is represented by a $k$-parameter family of embeddings of $\D^1$ into $M$ which cannot be trivialised through such families. Second, recall from Example~\ref{ex:Lie-pi} that we can think of decorated trees as homotopy classes of maps of trees into $M$.

A \emph{meridian} of $\u$ is a sphere $\S^{d-2}\hra M\sm\u(\D^1)$ defined as the boundary of a normal disk at some point $p\in\u(\D^1)$. We can form the Samelson bracket $\S^{n(d-3)}\to\Omega(M\sm\u(\D^1))$ of $n$ disjoint meridians, whose adjoint is the Whitehead bracket $W\colon\S^{n(d-3)+1}\to M\sm\u(\D^1)$. We hope to use this map to obtain an $n(d-3)$-parameter family of embeddings of $\D^1$ based at $\u$. Namely, if we could foliate $W(\S^{n(d-3)+1})$ by embedded circles $\mu_{\vec{t}}$, we could then ambiently connect sum a fixed subarc of $\u$ into each of $\mu_{\vec{t}}$, using group elements $g_i$. Such a family indeed exists, as we showed in \cite{K-families} and will summarise now.

The first ingredient is the following result, that constructs a universal such family for links. Write $\S^1=\S^1_-\cup\S^1_+$ as the union of the northern and southern semicircles. Let 
\[
    a_i\colon\D^1\hra\ball^d
\]
be disjoint neatly embedded arcs, $1\leq i\leq n$ (note that this is unique up to isotopy for $d\geq4$). Moreover, fix an arc $a_0\colon\D^1\hra\partial\ball^d$. Let $m_i\colon\S^{d-2}\hra\ball^d\sm\bigsqcup_{i\in\ul{n}}a_i$ be the meridian of $a_i$ at $a_i(0)$ in $\ball^d$, and assume they are mutually disjoint, and connected by paths, called whiskers, to the basepoint $a_0(0)$. Recall that for a Lie word $w$ and maps $x_i\colon\S^{d-3}\to\Omega\S^{d-2}$ adjoint to the identity, the map $w(\{x_i\})\colon \S^{n(d-3)}\to\Omega\bigvee_{i\in\ul{n}}\S^{d-2}_i$ is the canonical Samelson map~\eqref{eq:can-Sam}. 

\begin{thm}[{\cite[Definitions~5.5-5.10, Theorem~6.6]{K-families}}]
\label{thm:emb-comm}
    For any $d\geq3$ there exists a map $\wt{w}(\{m_i\})$ so that the following square commutes up to homotopy
\[
\begin{tikzcd}
    \S^{n(d-3)}
    \dar[swap]{w(\{x_i\})}\rar{\wt{w}(\{m_i\})} 
    & \Emb\Big((\S^1,\S^1_+),(\ball^d\sm\bigsqcup_{i\in\ul{n}} a_i,a_0)\Big)
    \dar[hook]\\
    \Omega\bigvee_{i\in\ul{n}}\S^{d-2}_i
    \rar[hook]{\Omega\bigvee_{i\in\ul{n}} m_i}[swap]{\sim}
    & \Omega\Big(\ball^d\sm\bigsqcup_{i\in\ul{n}} a_i\Big).
\end{tikzcd}
\]
\end{thm}

The second ingredient translates these families $\wt{w}(\{m_i\})$ into embeddings $\unknot\colon \D^1\hra M$ as follows.
\begin{defn}
\begin{itemize}
\item 
    A \emph{grasper} of degree $n$ is an embedding $\TG\colon\ball^d\hra M$, which is disjoint from $\partial M$, and which intersects the basepoint knot $\u$ in fixed intervals
    \[
        \TG(\ball^d)\cap\u(\D^1)=\bigsqcup_{i=0}^n\u(J_i)=\bigsqcup_{i=0}^n\TG(a_i).
    \]
\item
    For a grasper $\TG\colon\ball^d\hra M$ and $0\leq i\leq n$ denote by $p_i\in \D^1$ the point such that $\u(p_i)=\TG(a_i(0))$. We define $g_i\in\pi_1M$ as the homotopy class of the loop given as the concatenation of two paths between $\u(p_0)$ and $\u(p_i)$: one is $\u|_{[p_0,p_i]}$, and the other is contained in $\TG(\ball^d)$. We say that the tuple $g_{\ul{n}}\coloneqq(g_1,\dots,g_n)\in(\pi_1M)^n$ \emph{decorates} the grasper $\TG$, and write $\TG_{g_{\ul{n}}}$ instead.
\item
    The \emph{grasper surgery of type $w$ along $\TG_{g_{\ul{n}}}$} is the family of embeddings
    \begin{align*}
       \wt{w}(\TG_{g_{\ul{n}}})\colon\S^{n(d-3)} & \ra\Embp(\D^1,M),\\ 
        \vec{\tau}\; &\mapsto \quad \u|_{\D^1\sm J_0}\cup \big(\TG\circ \wt{w}(m_i)(\vec{\tau})|_{\S^1_-}\big).\qedhere
    \end{align*}
\end{itemize}
\end{defn}

This generalises to all embedding spaces with 1-dimensional source constructions from classical knot theory: the Gusarov--Habiro clasper surgery \cite{Gusarov,Habiro} and Conant--Teichner grope cobordism~\cite{CT1}. In fact, the name comes from \emph{grasper=grope+clasper=grasp+er}. Both are closely related to iterated commutators in the fundamental group of the complement of a knot (its lower central series). A grope contains a clasper as a subset, but also carries additional data. In our generalisation this data is given by the appropriate choice of $\wt{w}(\{m_i\})$ in Theorem~\ref{thm:emb-comm}, and of family of its nullhomotopies inside of $\ball^d$, and is used to show that grasper classes vanish in $T_n(M)$.

Recall that $\Lie(n)$ is a quotient of $\Z[\Tree(n)]$ by the relations \eqref{eq:L-as-ihx}. Define the homomorphism 
    \[
        \realmap_n\colon
    \Z[\Tree(n)\otm(\pi_1M)^n]\ra\pi_{n(d-3)}\Embp(\D^1,M)
    \]
by linearly extending the map $(w, g_{\ul{n}})\mapsto [\wt{w}(\TG_{g_{\ul{n}}})]$,
where $\TG$ is any grasper of degree $n$ in $M$ relative to $\u$ and decorated by $g_{\ul{n}}$.
By \cite{K-families} this vanishes on the relations $\AS,\IHX$, and precisely realises the classes in the layer --- see Theorem~\ref{thm:K-families} in the introduction. That result says that on one hand, the family $\wt{w}(m_i)$ of re-embeddings of the arc $a_0$ can be viewed as the \emph{universal iterated Whitehead bracket} of the shape $w$ and dimension $d$, whereas $\realmap_n(w,g_{\ul{n}})$ is an \emph{embedded version of the Whitehead bracket} according to $(w,{g_{\ul{n}}})$ of meridian spheres for $\u$. Moreover, these classes precisely correspond to canonical Samelson maps in the layers of the tower.

%%%%%%%%%%

\printbibliography

\vspace{5pt}
\hrule

\end{document}